\documentclass[a4paper,12pt]{article}

\usepackage{amsmath,amssymb,amsthm}
\usepackage{amsfonts}

\usepackage[table]{xcolor}
\PassOptionsToPackage{table}{xcolor}

\usepackage{tikz-cd}

\usepackage{url}
\usepackage{here}
\usepackage{enumitem}
\usepackage{a4wide}
\usepackage{hyperref}

\usepackage{authblk}
\newtheorem{theorem}{Theorem}[section]
\newtheorem{lemma}[theorem]{Lemma}
\newtheorem{proposition}[theorem]{Proposition}
\newtheorem{corollary}[theorem]{Corollary}

\newtheorem{question}[theorem]{Question}

\newtheorem{remark}[theorem]{Remark}
\newtheorem{definition}[theorem]{Definition}

\newcommand{\RR}{\mathbb R}
\newcommand{\PP}{\mathbb P}

\newcommand{\QQ}{\mathbb Q}
\newcommand{\CC}{\mathbb C}

\newcommand{\Eff}{\operatorname{Eff}}

\newcommand{\Nef}{\operatorname{Nef}}

\newcommand{\Pic}{\operatorname{Pic}}

\newcommand{\Bs}{\operatorname{Bs}}

\newcommand{\arrow}{\rightarrow}

\newcommand{\vdim}{\operatorname{vdim}}
\newcommand{\edim}{\operatorname{edim}}

\title{Log Fano blowups of mixed products of projective spaces and their effective cones}
\date{}
\author[1]{Tim Grange\thanks{Current address: School of Mathematical Sciences, University of Nottingham. Email: \href{mailto:tim.grange@nottingham.ac.uk}{\nolinkurl{tim.grange@nottingham.ac.uk}}}}
\author[2]{Elisa Postinghel\thanks{\href{mailto:elisa.postinghel@unitn.it}{\tt elisa.postinghel@unitn.it}}}
\author[1]{Artie Prendergast-Smith\thanks{\href{mailto:a.prendergast-smith@lboro.ac.uk}{\tt a.prendergast-smith@lboro.ac.uk}}}
  \affil[1]{Department of Mathematical Sciences, Loughborough University}%, Epinal Way, Loughborough, LE11 3TU, United Kingdom}}
  \affil[2]{Dipartimento di Matematica, Universit\`a degli Studi di Trento}%, 
\makeatletter
\newcommand{\subjclass}[2][1991]{%
  \let\@oldtitle\@title%
  \gdef\@title{\@oldtitle\footnotetext{#1 \emph{Mathematics Subject Classification.} #2}}%
}
\newcommand{\keywords}[1]{%
  \let\@@oldtitle\@title%
  \gdef\@title{\@@oldtitle\footnotetext{\emph{Key words and phrases.} #1.}}%
}
\makeatother

\begin{document}



\maketitle

\begin{abstract}
We compute the cones of effective divisors on blowups of $\PP^1 \times \PP^2$ and $\PP^1 \times \PP^3$ in up to 6 points. We also show that all these varieties are log Fano, giving a conceptual explanation for the fact that all the cones we compute are rational polyhedral.
\end{abstract}

Cones of divisors are combinatorial objects that encode key geometric information about algebraic varieties. Understanding the structure of these cones is therefore a basic problem in algebraic geometry. There are important theorems describing the structure of these cones for important classes of varieties: for example, Fano varieties have rational polyhedral nef cones by the Cone Theorem, and rational polyhedral effective cones by Birkar--Cascini--Hacon--McKernan \cite{BCHM}.

In general, however, these cones can be difficult to understand, even for simple varieties such as blowups of sets of points in projective space or products of projective spaces. Mukai \cite{Mukai05} and Castravet--Tevelev \cite{CT06} proved that the blowup of $\left(\PP^n\right)^r$ in a set of $s$ points in very general position is a Mori dream space, and in particular has rational polyhedral effective cone, if and only if
\begin{align*}
  \frac{1}{r+1}+\frac{1}{s-n-1}+\frac{1}{n+1}>1.
\end{align*}
These blowups are highly symmetric from the point of view of divisor theory: there is a naturally-defined Weyl group action on the Picard group, and this is the key ingredient in describing the effective cones and Cox ring when they are finitely generated.

In this paper, we focus instead on certain ``asymmetric'' varieties of blowup type: namely, blowups of $\PP^1 \times \PP^2$ and $\PP^1 \times \PP^3$ in sets of up to 6 points in general position. The presence of factors of different dimensions means that the Picard groups of these varieties are less symmetric than in the previous examples, so we use different ideas to understand divisors. Our method combines the following techniques: {\it induction}, that is, pulling back divisors from lower blowups to get information about higher blowups; {\it restriction}, that is, obtaining necessary conditions for effectivity by restriction to suitable subvarieties whose cone of effective divisors is understood; and {\it base locus lemmas}, showing that divisors violating a numerical inequality must contain a certain fixed divisor as a component.

This method, that we shall refer to as the \emph{cone method}, is outlined in Section \ref{effconemethod}.  Our main results, contained in Sections \ref{section-dim3}-\ref{section-dim4}, give explicit descriptions of the cones of effective divisors on these varieties and describe the geometry of the generating classes. For each such cone, we will give a list of extremal rays and a list of inequalities cutting it out: the latter corresponds to giving extremal rays of the corresponding cone of moving curves. Along the way, we also compute the effective cones of divisors of some threefolds given as blowups of $\PP^3$ along a line and up to six points in general position (Section \ref{section-dim3-extra}).

More conceptually, one can also ask when varieties of the above kind are log Fano. For blowups of $\PP^2$, being Fano is equivalent to being a Mori dream space and for these cases the Cox rings were described by Batyrev--Popov in \cite{BP04}. More recently, Araujo--Massarenti \cite{AM16} and Lesieutre--Park \cite{LP17} proved that the same holds in higher dimension, namely that blowups of $\PP^n$ or of products of the form $\left(\PP^n\right)^m$ in points in very general position are log Fano if and only if they are Mori dream spaces.

The same questions are open for mixed products, i.e. for blowups of $\PP^{n_1}\times\cdots\times\PP^{n_m}$ with not all $n_k$ equal. For blowups of $\PP^1 \times \PP^2$ and of $\PP^1 \times \PP^3$ in up to 6 points in general position, we show in Section \ref{weak-log} that all of these varieties are log Fano, and therefore Mori dream spaces. We also use the results of Mukai and Castravet--Tevelev to deduce that the blowup of $\PP^1 \times \PP^n$ in sufficiently many points is not a Mori dream space; Theorem \ref{theorem-summary} summarises what we know in this direction. This leaves a small number of open cases in each dimension, which we collect in Questions \ref{question1}--\ref{question3}.

%% {\color{blue} More precisely, we obtain the following result. 
%% \begin{theorem}\label{thmMDS}
%% The blow-up of $\PP^1 \times \PP^n$ in $s$ points in very general position is a Mori dream space if $s\le 6$. The blow-up of $\PP^1 \times \PP^2$ (resp. $\PP^1 \times \PP^3$) in $s$ points is not a Mori dream space if $s\ge9$ (resp. $s\ge8$).
%% \end{theorem}
%% }

{\bf Acknowledgements:}  
The second author is a member of INdAM-GNSAGA and she was partially supported by the EPSRC grant EP/S004130/1. The authors thank Hamid Abban, Izzet Coskun, and Hendrik S\"{u}{\ss} for valuable suggestions and corrections.

\section{Preliminaries}\label{preliminaries}
We work throughout over the complex numbers $\CC$. 

For a variety $X$, we write $N^1(X)$ to denote the group of Cartier divisors on $X$ modulo numerical equivalence, tensored with $\RR$. This is a finite-dimensional real vector space whose dimension is called the {\it Picard rank} of $X$, and denoted $\rho(X)$. Dually, we write $N_1(X)$ for the group of $1$-cycles on $X$ modulo numerical equivalence, tensored with $\RR$. When $X$ is smooth, intersection of divisors and curves gives a perfect pairing $N^1(X) \times N_1(X) \arrow \RR$. Where appropriate, we will use the same symbol to denote a divisor and its class in $N^1(X)$, or a curve and its class in $N_1(X)$. We write $\Eff(X)$ to denote the cone in $N^1(X)$ spanned by classes of effective divisors; in general this cone need not be open or closed, but in all our examples it will be rational polyhedral, in particular closed.

The main objects of interest in this paper will be blowups of products of two projective spaces in a collection of points in general position. A statement holds for a collection of points in {\it general position}, respectively {\it very general position}, if it holds when the corresponding element in the Hilbert scheme of $s$ points of $\PP^m \times \PP^n$ lies in the complement of a proper Zariski closed subset, respectively in the complement of a countable union of Zariski closed subsets. For convenience we fix the following notation:
\begin{itemize}
\item $X_{m,n,s}$: the blowup of $\PP^m \times \PP^n$ in a set of $s$ points in general position $\left\{p_1,\ldots,p_s\right\}$;
\item $\pi_m, \, \pi_n$: the natural morphisms $X_{m,n,s} \arrow \PP^m$ and $X_{m,n,s} \arrow \PP^n$ respectively; by abuse of notation, we will use the same symbols to denote the corresponding morphisms $\PP^m \times \PP^n \arrow \PP^m$ and $\PP^m \times \PP^n \arrow \PP^n$;
  \item $H_1, \, H_2$: the pullbacks of the hyperplane classes on $\PP^m$ and $\PP^n$ via $\pi_m$ and $\pi_n$ respectively;
\item $l_1, \, l_2$: the classes of a line contained in a $\PP^m$-fibre of $\pi_n$, respectively a line contained in a $\PP^n$-fibre of $\pi_m$;
  \item $E_i, \, e_i$: the exceptional divisor of the blowup of the point $p_i$, respectively the class of a line contained in $E_i$.
\end{itemize}

The following proposition records the intersection numbers we need. All statements are straightforward consequences of general results about intersection theory of blowups; a reference is \cite[Proposition 13.12]{EH}.
\begin{proposition}\label{intersection-table}
  Let $X_{m,n,s}$ be as above. Then
  \begin{enumerate}
  \item[(a)] The vector spaces $N^1(X_{m,n,s})$ and $N_1(X_{m,n,s})$ have the following bases:
  \begin{align*}
    N^1(X_{m,n,s}) &= \langle H_1, H_2, E_1, \ldots, E_s \rangle, \\
    N_1(X_{m,n,s}) &= \langle l_1, l_2, e_1, \ldots, e_s \rangle.
  \end{align*}
\item[(b)] We have the following intersection numbers among divisors:
  \begin{align*}
    H_1^m \cdot H_2^n &=1,\\
    H_1^p \cdot H_2^{m+n-p} &=0 \quad \text{ for } p \neq m, \\
    H_1^p \cdot E_i^{m+n-p} = H_2^p \cdot E_i^{m+n-p} &=0 \ \text{ for all } p >0, \ i=1,\ldots,s,\\
    E_i^{m+n} &= (-1)^{m+n-1}.
  \end{align*}
\item[(c)] We have the following intersection numbers between divisors and curves:
\begin{align*}
  H_i \cdot l_j &= \delta_{ij},\\
  H_i \cdot e_j &= 0 \quad \text{ for } i=1,2, \, j=1,\ldots,s,\\
  E_i \cdot e_j &= -\delta_{ij}.
\end{align*}
  \end{enumerate}
\end{proposition}

In particular this allows us to determine the numerical classes of curves on blowups, as follows.
\begin{corollary}\label{corollary-curveclass} Let $C$ be the proper transform on $X_{m,n,s}$ of a curve of bidegree $(d_1,d_2)$ in $\PP^m \times \PP^n$ with multiplicity $m_i$ at the point $p_i$. Then the class of $C$ in $N_1(X_{m,n,s})$ is
  \begin{align*}
    d_1l_1 + d_2l_2 - \sum_{i=1}^s m_i e_i.
  \end{align*}
\end{corollary}
  \begin{proof}
    We have the intersection numbers
    \begin{align*}
      C \cdot H_i &= l_i \quad  \text{ for } i=1,2,\\
      C \cdot E_j &= m_j \quad \text{ for } j=1,\dots, s.
    \end{align*}
    Since the intersection pairing on $X_{m,n,s}$ is perfect, the given formula then follows from Proposition \ref{intersection-table} (a) and (c).
  \end{proof}
In particular we deduce the following formula, which we use in Section \ref{weak-log}. For any divisor $D$ on $X_{m,n,s}$, the class of $D$ can be written in the form
\begin{equation}\label{general divisor}
  d_1H_1+d_2H_2 - \sum_{i=1}^s m_i E_i,
\end{equation}
for some integers $d_1, \, d_2, \, m_1, \ldots, m_s$. 
\begin{corollary}\label{corollary-topselfint}
  On the variety $X_{m,n,s}$ consider a divisor $D$ with class \eqref{general divisor}.
  Then the top self-intersection number of $D$ is given by
  \begin{align*}
    D^{m+n} &= d_1^m d_2^n {m+n \choose n} -\sum_{i=1}^s m_i^{m+n}.
  \end{align*}
\end{corollary}

\subsection*{Virtual dimension and expected dimension}

Given a divisor $D$ on $X_{m,n,s}$, we can give a lower bound for the dimension of the linear system $|D|$, obtained by a simple parameter count. 

\begin{definition}\label{virtual dimension} 
Let $D$ be a divisor on $X_{m,n,s}$ with class \eqref{general divisor}.
The \emph{virtual dimension} of the linear system $|D|$ 
is the integer
$$
\vdim|D|={{m+d_1}\choose m}{{n+d_2}\choose n}-\sum_{i=1}^s{{m+n+m_i-1}\choose {m+n}}-1.$$
The \emph{expected dimension} of $|D|$ is
$$
\edim|D|=\max\{-1,\vdim|D|\}.
$$
\end{definition}

\begin{lemma}\label{vdim lower bound}
We have $\dim|D|\ge\edim|D|$.
\end{lemma}
\begin{proof}
If $d_i<0$, $\vdim|D|<0$ and $|D|$ is empty, so the statement holds. If $m_i\le0$ then $\vdim|D|=\vdim|D+m_iE_i|$
and
$|D|=(-m_i)E_i+|D+m_iE_i|$, so the statement holds for $D$ if it holds for all divisors with $m_i \geq 0$. 

Assume now that $d_1,d_2,m_i\ge0$ for every $1\le i\le s$.
Notice that a bidegree $(d_1,d_2)$ hypersurface of $\PP^m\times\PP^n$ is the zero locus of a bihomogeneous polynomial $F_{(d_1,d_2)}$ in $n+m+2$ variables that depends on ${{m+d_1}\choose m}{{n+d_2}\choose n}$ parameters. For such a hypersurface, the passage through a point with multiplicity $m_i$ corresponds to the vanishing of all partial derivatives of $F_{(d_1,d_2)}$ of order $m_i-1$ and therefore it imposes  ${{m+n+m_i-1}\choose {m+n}}$ linear conditions. Since elements of the linear system $|D|$ are in $1-1$ correspondence with the elements of the projectivisation of the vector space of such polynomials, we can conclude.
\end{proof}
Later we will use the following description of the virtual dimension. This lemma is well known to experts but we are not aware of a reference.
\begin{lemma}\label{vdimeuler}
  For $D$ a divisor of the form \eqref{general divisor} with $m_i \geq 0$ for all $i$, the following relation between virtual dimension and Euler characteristic holds:
  \begin{align*}
    \vdim|D| &= \chi(X,O_X(D))-1.
  \end{align*}
\end{lemma}

\begin{proof}
  Let $D$ be a divisor with class $d_1H_1+d_2H_2-\sum_i m_iE_i$ where $d_1,d_2$, and the $m_i$ are all nonnegative integers. We will prove the claim by induction on the natural number $M=\sum_i m_i$.

  If $M=0$ then $D=d_1H_1+d_2H_2$ and both sides of the claimed equality equal
  \begin{align*}
    {m+d_1 \choose m} {n + d_2 \choose n}-1.
  \end{align*}
  Now assume the equality holds for $M=k-1$; we will prove it for $M=k$. Let $D$ be a divisor of the form \eqref{general divisor} with $M=k$. Assume without loss of generality that $m_1>0$, and define $D^\prime = D+E_1$: note that the divisor $D^\prime$ has $M=k-1$, so by our induction hypothesis we have $\vdim|D^\prime|=\chi(X,O_X(D^\prime))-1$.

  On one hand we have
  \begin{align*}
    \vdim|D^\prime|-\vdim|D| &= {m+n+m_1-1 \choose m+n} - {m+n+m_1-2 \choose m+n} \\
    &=   {m+n+m_1-2 \choose m+n-1 }.
  \end{align*}
  On the other hand, we can twist the ideal sheaf sequence for $E_1$ by $O_X(D^\prime)$ to get the exact sequence
  \begin{align*}
    0 \arrow O_X(D) \arrow O_X(D^\prime) \arrow O_{E_1}(D^\prime|_{E_1}) \arrow 0
  \end{align*}
  which gives
  \begin{align*}
    \chi(X,O_X(D^\prime)) - \chi(E_1, O_{E_1}(D^\prime|_{E_1})) &=  \chi(X,O_X(D)).
  \end{align*}
  Now $E_1 \cong \PP^{m+n-1}$ and $D^\prime|_{E_1} \cong O_{\PP^{m+n-1}}(m_1-1)$ which implies
  \begin{align*}
    \chi(E_1, O_{E_1}(D^\prime|_{E_1}))
    &= {m+n+m_1-2 \choose m+n-1}.
  \end{align*}
Using these expressions and the induction hypothesis we get
  \begin{align*}
    \vdim|D| &= \vdim|D^\prime| -   {m+n+m_1-2 \choose m+n-1 }\\
    &= \chi(X,O_X(D^\prime))-1 -   {m+n+m_1-2 \choose m+n-1 }\\
    &= \chi(X,O_X(D^\prime))-1-\chi(E_1, O_{E_1}(D^\prime|_{E_1}))\\
    &= \chi(X,O_X(D))-1.    
  \end{align*}
\end{proof}

\subsection*{Base locus lemmas}
An important ingredient in our computations will be base locus lemmas showing that effective divisor classes which violate certain numerical inequalities must contain a certain fixed divisor in their base locus. The following lemma gives a convenient way to prove several of these results.

\begin{lemma}\label{lemma-baselocusgeneral}
  Let $X$ be a smooth projective variety. Let $C \in N_1(X)$ be a curve class, and let $F$ be an irreducible reduced divisor on $X$ such that irreducible curves with class $C$ cover a Zariski-dense subset of $F$. Let $f=F \cdot C$, and assume that $f<0$.

  If $D$ is an effective divisor class on $X$ such that
  \begin{align*}
    D \cdot C = d <0
  \end{align*}
  then the divisor $F$ is contained in the base locus $\Bs(D)$ with multiplicity at least $\left \lceil{d/f} \right \rceil$. 
\end{lemma}
\begin{proof}
Since $D \cdot C<0$, every curve in the class $C$ must be contained in every effective divisor in the class $D$, so $F \subset \Bs(D)$. Now replace $D$ with $D-F$ and continue: the same argument applies to $D-kF$ as long as $d-kf<0$.
\end{proof}
In particular we immediately deduce base locus lemmas for two kinds of fixed divisors on our varieties:
\begin{lemma}\label{lemma-baselocusexceptionals} On the variety $X_{m,n,s}$ consider an effective divisor $D$ with class
$
    d_1H_1+d_2H_2 - \sum_{i=1}^s m_i E_i
$
  Then:
  \begin{enumerate}
  \item[(a)] for each $i$, the exceptional divisor $E_i$ is contained in $\Bs(D)$ with multiplicity at least $\operatorname{max}(0,-m_i)$;
 \item[(b)] in the case $m=1$, the unique effective divisor $F_i$ with class $H_1-E_i$ is contained in $\Bs(D)$ with multiplicity at least $\operatorname{max}(0,m_i-d_2)$.
  \end{enumerate}
  \end{lemma}
\begin{proof} To prove $(a)$, apply Lemma \ref{lemma-baselocusgeneral} with $C=e_i$, the class of a line in $E_i$. Note that $E_i \cdot e_i = -1$.

  To prove $(b)$, apply Lemma \ref{lemma-baselocusgeneral} with $C=l_2-e_i$. This is the class of the proper transform on $X_{1,n,s}$ of any line in a fibre of $\PP^1 \times \PP^n \arrow \PP^1$ which passes through $p_i$. These lines cover the fibre, hence their proper transforms cover the proper transform of the fibre. Note that $(H_1-E_i)\cdot(l_2-e_i)=-1$.  
\end{proof}

\subsection*{Effective cones of del Pezzo surfaces}

In this subsection we record the effective cones of various del Pezzo surfaces. These cones are described in standard references such as \cite{manin}. For this paper, it will be convenient to view these surfaces as blowups of $\PP^1 \times \PP^1$ rather than of $\PP^2$. In our notation, $X_{1,1,s}$ denotes a del Pezzo surface of degree $8-s$. 

\begin{proposition}\label{proposition-x115} 
The effective cone $\Eff(X_{1,1,5})$ is given by the generators below left and the inequalities below right.

\begin{minipage}[t]{0.4\textwidth}
 
\begin{align*}
\rowcolors{2}{white}{gray!15}
\begin{array}{ccccccc}
H_1 & H_2 & E_1 & E_2 & E_3 & E_4 & E_5\\
\hline\hline
%\hhline{======}
0 & 0 & 1 & 0 & 0 & 0 & 0\\
1 & 0 & -1 & 0 & 0 & 0 & 0\\
1 & 1 & -1 & -1 & -1 & 0 & 0\\
2 & 1 & -1 & -1 & -1 & -1 & -1\\
\end{array}
\end{align*}
\end{minipage}\quad
\begin{minipage}[t]{0.4\textwidth}

\begin{align*}
\rowcolors{2}{white}{gray!15}
\begin{array}{ccccccc}
d_1 & d_2 & m_1 & m_2 & m_3 & m_4 & m_5\\
\hline\hline
%\hhline{======}
1 & 0 & 0 & 0 & 0 & 0 & 0\\
1 & 1 & 1 & 0 & 0 & 0 & 0\\
1 & 1 & 1 & 1 & 0 & 0 & 0\\
1 & 2 & 1 & 1 & 1 & 0 & 0\\
1 & 2 & 1 & 1 & 1 & 1 & 0\\
1 & 3 & 1 & 1 & 1 & 1 & 1\\
2 & 2 & 2 & 1 & 1 & 1 & 0\\
2 & 2 & 2 & 1 & 1 & 1 & 1\\
2 & 3 & 2 & 2 & 1 & 1 & 1\\
3 & 3 & 2 & 2 & 2 & 2 & 1\\
\end{array}
\end{align*} 
\end{minipage}
\end{proposition}

\begin{proposition}\label{proposition-x116} 
The effective cone $\Eff(X_{1,1,6})$ is given by the generators below left and the inequalities below right.

\begin{minipage}[t]{0.4\textwidth}
 
  \begin{align*}
\rowcolors{2}{white}{gray!15}
\begin{array}{cccccccc}
H_1 & H_2 & E_1 & E_2 & E_3 & E_4 & E_5 & E_6\\
\hline\hline
%\hhline{======}
0 & 0 & 1 & 0 & 0 & 0 & 0 & 0\\
1 & 0 & -1 & 0 & 0 & 0 & 0 & 0\\
1 & 1 & -1 & -1 & -1 & 0 & 0 & 0\\
2 & 1 & -1 & -1 & -1 & -1 & -1 & 0\\
2 & 2 & -2 & -1 & -1 & -1 & -1 & -1\\
\end{array}
\end{align*} 
\end{minipage}\quad
\begin{minipage}[t]{0.4\textwidth}

\begin{align*}
\rowcolors{2}{white}{gray!15}
\begin{array}{cccccccc}
d_1 & d_2 & m_1 & m_2 & m_3 & m_4 & m_5 & m_6\\
\hline\hline
%\hhline{======}
1 & 0 & 0 & 0 & 0 & 0 & 0 & 0\\
1 & 1 & 1 & 0 & 0 & 0 & 0 & 0\\
1 & 1 & 1 & 1 & 0 & 0 & 0 & 0\\
1 & 2 & 1 & 1 & 1 & 0 & 0 & 0\\
1 & 2 & 1 & 1 & 1 & 1 & 0 & 0\\
1 & 3 & 1 & 1 & 1 & 1 & 1 & 0\\
1 & 3 & 1 & 1 & 1 & 1 & 1 & 1\\
2 & 2 & 2 & 1 & 1 & 1 & 0 & 0\\
2 & 2 & 2 & 1 & 1 & 1 & 1 & 0\\
2 & 3 & 2 & 2 & 1 & 1 & 1 & 0\\
2 & 3 & 2 & 2 & 1 & 1 & 1 & 1\\
2 & 4 & 2 & 2 & 2 & 1 & 1 & 1\\
3 & 3 & 2 & 2 & 2 & 2 & 1 & 0\\
3 & 3 & 2 & 2 & 2 & 2 & 1 & 1\\
3 & 3 & 3 & 2 & 1 & 1 & 1 & 1\\
3 & 4 & 2 & 2 & 2 & 2 & 2 & 2\\
3 & 4 & 3 & 2 & 2 & 2 & 1 & 1\\
3 & 5 & 3 & 2 & 2 & 2 & 2 & 2\\
4 & 4 & 3 & 3 & 2 & 2 & 2 & 1\\
4 & 5 & 3 & 3 & 3 & 2 & 2 & 2\\
5 & 5 & 3 & 3 & 3 & 3 & 3 & 2\\
\end{array}
\end{align*}
\end{minipage}
\end{proposition}
\noindent {\bf Writing cones } Let us spell out the conventions we use in writing effective cones in the tables in Propositions \ref{proposition-x115} and \ref{proposition-x116} above, and subsequently in this paper. Each effective cone we compute is described in two equivalent ways: by a list of generating classes, and by a list of defining inequalities.

For the table giving the list of generators, a row of the table of the form $(a_1 \ a_2 \ b_1 \cdots b_n)$ corresponds to a generator $a_1H_1+a_2H_2 + \sum_i b_i E_i$ of the effective cone.  For example, in the tables of Proposition \ref{proposition-x116} above, Row 3 of the left table corresponds to the generator \begin{align*}
  H_1+H_2-E_1-E_2-E_3
\end{align*}
For the tables giving the list of defining inequalities, a row with entries $(\alpha_1 \ \alpha_2 \ \beta_1 \cdots \beta_n)$ corresponds to an inequality $\alpha_1 d_1 + \alpha_2 d_2 \geq \sum_i \beta_ i m_i$ which must be satisfied by an effective divisor written in the form (\ref{general divisor}). So for example in Proposition \ref{proposition-x116} above, Row 3 of the right table above corresponds to an inequality
\begin{align*}
  d_1+d_2 \geq m_1+m_2 
\end{align*} 
A key point to stress is that these lists should be read {\bf up to permutation}: for a given generator written in the list, all generators obtained by suitable permutations are also included in the list, and similarly for inequalities. So for example, as well as the class and inequality written above, the lists of generators and inequalities respectively for $\Eff(X_{1,1,6})$ also include all of the following:
\begin{align*}
  H_1+H_2-E_i-E_j-E_k,\\
  d_1+d_2\ge m_i+m_j,
\end{align*}
where $i, \, j, \, k$ are distinct indices in the set $\{1,\ldots,6\}$.
%{\color{blue}
  Finally, if $m=n$ such as in Propositions \ref{proposition-x115} and \ref{proposition-x116},
also $H_1,H_2$ (resp. $d_1,d_2$)  can be swapped in the list of generators (resp. inequalities).%}

\section{The cone method}\label{effconemethod}

Let $X_{m,n,s}$ be the blowup of $\PP^n\times\PP^m$ in $s$ points in general position. In this section, we outline a method that aims to determine $\Eff(X_{m,n,s})$ from several pieces of information: knowledge of the effective cone $\Eff(X_{m,n,s-1})$, base locus lemmas for fixed divisors, and the effective cones of certain subvarieties. In later sections, this method will be applied recursively to determine the effective cones of $X_{1,n,s}$ for $n=2, \, 3$ and $s \leq 6$.  
 
Using the notation of Section \ref{preliminaries}, we write a divisor $D$ in $X_{m,n,s}$ as $$
D=d_1H_1+d_2H_2-\sum_{i=1}^sm_iE_i,$$
where $d_1,d_2,m_1,\dots,m_s$ are integers. 

The method consists of the following steps. 
\begin{enumerate}[leftmargin=+.65in]
\item[\textbf{Step 1}] First of all, notice that $D+m_iE_i$ can be thought of as a divisor on $X_{m,n,s-1}$.
Assume that 
we know the inequalities
cutting out the effective cone of $\Eff(X_{m,n,s-1})$. These inequalities 
form a set of necessary and sufficient conditions,  in  the variables $d_1,d_2$, $m_1\dots,m_{i-1},m_{i+1},\dots,m_{s}$, for $D+m_iE_i$ to be effective on $X_{m,n,s-1}$, and a set of necessary conditions for $D$ to be effective on $X_{m,n,s}$, thanks to the obvious relation $D+m_iE_i\ge D$.
Therefore, letting $i$ vary between $1$ and $s$, we obtain a list of inequalities in $d_1,d_2,m_1,\dots,m_{s}$ giving rise to a cone that contains $\Eff(X_{m,n,s})$.
\item[\textbf{Step 2}] 
Secondly, for an extremal ray of the cone obtained in \textbf{Step 1} that is spanned by a fixed divisor $F$, we will compute a \emph{base locus lemma} for $F$, namely we will give an integer $K_F(D)$, depending on $F$ and on $D$, such that, if positive,  $F$ is contained in the base locus of $D$ at least $K_F(D)$ times.  Then we shall add the inequality $K_F(D)\le0$ to the list of  \textbf{Step 1}, giving rise to a smaller cone. Notice that all effective divisors that are excluded by the new inequality must contain $F$ as a component.
\item[\textbf{Step 3}] 
Finally, we compute the ray generators of the cone determined in \textbf{Step 2}, potentially identifying  new extremal rays.
If we can prove that the latter are all effective, then we add to the list the rays generated by the fixed divisors $F$ excluded in \textbf{Step 2}, and we obtain a complete list of generators of $\Eff(X_{m,n,s})$.

\item[\textbf{Step 4}] 
When the output of \textbf{Step 3} has some non effective extremal rays,  we need to refine the starting cone by adding additional necessary conditions for  the effectivity of a general divisor $D$. The key idea here is that if $M$ is a movable divisor on $X_{m,n,s}$ and if $|D|_M|=\emptyset$, then $|D|=\emptyset$ too. Therefore the inequalities describing the effective cone of the divisor $M$ will yield necessary conditions for $D$ to be effective. These conditions will be added to those from  \textbf{Step 1}. This will have the effect of shrinking the starting cone, and we run the procedure again. 
\end{enumerate}
 
Computationally, we implement this method for a given variety using the software package Normaliz \cite{normaliz}. For each variety $X_{m,n,s}$ that we consider, our set of supplementary files contains 4 Normaliz files with filenames of the form {\tt Xmns-*.*} and with the following contents:
\begin{itemize}
\item in the input file {\tt Xmns-ineqs.in} we collect the inequalities obtained in \textbf{Step 1}, \textbf{Step 2} and \textbf{Step 4};
\item the output file {\tt Xmns-ineqs.out} then computes, among other things, a list of generators of the restricted cone cut out by all these inequalities;
\item in the input file {\tt Xmns-gens.in} we collect the generators of the restricted cone from the previous step, together with all known fixed divisor classes on $X_{m,n,s}$, following \textbf{Step 3};
  \item finally, the output file {\tt Xmns-gens.out} computes the output of \textbf{Step 3}, giving both a list of generators and a list of defining inequalities.  
\end{itemize}

In this article we will apply this method to compute the effective cones of the blowup of $X_{1,m,s}$, with $m=2,3$, $s\le 6$, but we believe this can be applied in more general settings.

%%%%%%%%%%%%%%%%%%%%%%
\section{Blowups of $\PP^1 \times \PP^2$}\label{section-dim3}

%%%%%%%%%%%%%%%%%%%%

In this section we consider the threefolds $X_{1,2,s}$, the blowup of $\PP^1\times\PP^2$ in $s$ points in general position.

\subsection*{1 or 2 points}

The blowup of $\PP^1\times\PP^2$ in one or two points is a smooth toric threefold, therefore the cones of divisors are generated by the boundary divisors.

\begin{theorem}\label{Eff 122}
The effective cone of $X_{1,2,1}$ is given by the list of generators below left and inequalities below right:

\begin{minipage}[t]{0.4\textwidth}
  \begin{align*}
\rowcolors{2}{white}{gray!15}
\begin{array}{ccc}
H_1 & H_2 & E_1  \\
\hline\hline
%\hhline{=====}
0 & 0 & 1  \\
1 & 0 & -1  \\
\end{array}
  \end{align*}
  \end{minipage}\quad
  \begin{minipage}[t]{0.4\textwidth}
    \begin{align*}
\rowcolors{2}{white}{gray!15}
\begin{array}{ccc}
d_1 & d_2 & m_1  \\
\hline\hline
%\hhline{=====}
1 & 0 & 0  \\
0 & 1 & 0  \\
1 & 1 & 1  \\
\end{array}
  \end{align*}
  \ 
\end{minipage}

The effective cone  of $X_{1,2,2}$ is given by the list of generators below left and inequalities below right:

\begin{minipage}[t]{0.4\textwidth}
  \begin{align*}
\rowcolors{2}{white}{gray!15}
\begin{array}{cccc}
H_1 & H_2 & E_1 & E_2 \\
\hline\hline
%\hhline{=====}
0 & 0 & 1 & 0 \\
1 & 0 & -1 & 0 \\
0 & 1 & -1 & -1 \\
\end{array}
  \end{align*}
  \end{minipage}\quad
  \begin{minipage}[t]{0.4\textwidth}
    \begin{align*}
\rowcolors{2}{white}{gray!15}
\begin{array}{cccc}
d_1 & d_2 & m_1 & m_2 \\
\hline\hline
%\hhline{=====}
1 & 0 & 0 & 0 \\
0 & 1 & 0 & 0 \\
1 & 1 & 1 & 0 \\
1 & 2 & 1 & 1 \\
\end{array}
  \end{align*}
\end{minipage}
\end{theorem}
\begin{proof}
The two tables on the left hand side contain the list of boundary divisors:  they span the extremal rays of the effective cone. The lists on the right hand side are computed by duality, we leave the details to the reader. \end{proof}

\begin{remark}\label{fixed}
All extremal rays of the effective cone of divisors of $X_{1,2,2}$ are fixed.  Exceptional divisors and divisors $H_1-E_i$ are fixed, cf. Lemma \ref{lemma-baselocusexceptionals}.
Finally,  consider the image of $p_i,p_j$ under the projection onto the second factor: $H_2-E_i-E_j$ is the pre-image of the line determined by these points and is therefore a fixed divisor isomorphic to the blowup of $\PP^1\times\PP^1$ at two  points in general position.
\end{remark}

\subsubsection*{Base locus lemmata for degree one divisors}\label{base locus degree 1}
In Lemma \ref{lemma-baselocusexceptionals}(b), we computed a lower bound for the multiplicity of containment of the fixed divisor $H_1-E_i$ in the base locus of an effective divisor. We now do the same for the fixed fivisor $H_2-E_i-E_j$.

\begin{lemma}\label{bsh2-ei-ej}
Let $D=d_1E_1+d_2E_2-\sum_{i=1}^s m_i E_i$ be an effective divisor, with $s\ge 2$.  The divisor $H_2-E_i-E_j$ is contained in $\Bs(D)$ with multiplicity at least $\max\{0,m_i+m_j-d_1-d_2\}$.
\end{lemma}

\begin{proof}

Without loss of generality, we may assume that $i=1,j=2$. 
Under the assumption that $m_i\ge 0$ for every $i=1,\dots,s$, it is enough to prove the statement for divisors on $X_{1,2,2}$, namely for $D=d_1H_1+d_2H_2-m_1E_1-m_2E_2$.
If $m_1+m_2-d_1-d_2\le0$ the statement is trivial, therefore we will assume that $m_1+m_2-d_1-d_2>0$.

Recall from Remark \ref{fixed}, that the divisor $F=H_2-E_1-E_2$ is isomorphic to $X_{1,1,2}$ and its Picard group is generated by the classes $l_1,l_2$ and $e_1,e_2$ (see Proposition \ref{intersection-table}). The surface is swept out by the irreducible curves with class $C=l_1+l_2-e_1-e_2$. We conclude using Lemma \ref{lemma-baselocusgeneral}, noticing that
\begin{align*}
F\cdot C& =-1\\ 
D\cdot C& =d_1+d_2-m_1-m_2.
\end{align*}
\end{proof}

\subsection*{3 points}

 The first interesting case is $s=3$ and it can be computed following the procedure outlined in Section \ref{effconemethod}.

\begin{theorem}\label{Eff 123}
The effective cone $\Eff(X_{1,2,3})$ is given by the list of generators below left and inequalities below right:

\begin{minipage}[t]{0.4\textwidth}
  \begin{align*}
\rowcolors{2}{white}{gray!15}
\begin{array}{ccccc}
H_1 & H_2 & E_1 & E_2 & E_3\\
\hline\hline
%\hhline{=====}
0 & 0 & 1 & 0 & 0\\
1 & 0 & -1 & 0 & 0\\
0 & 1 & -1 & -1 & 0\\
\end{array}
  \end{align*}
  \end{minipage}\quad
  \begin{minipage}[t]{0.4\textwidth}
    \begin{align*}
\rowcolors{2}{white}{gray!15}
\begin{array}{ccccc}
d_1 & d_2 & m_1 & m_2 & m_3\\
\hline\hline
%\hhline{=====}
1 & 0 & 0 & 0 & 0\\
0 & 1 & 0 & 0 & 0\\
1 & 1 & 1 & 0 & 0\\
1 & 2 & 1 & 1 & 0\\
1 & 2 & 1 & 1 & 1\\
\end{array}
  \end{align*}
\end{minipage}
\end{theorem}
\begin{proof}

We apply the method of Section \ref{effconemethod} (\textbf{Step 1}-\textbf{2}) with the following inputs:
  \begin{itemize}
   \item the pullback via the blowdown morphism $X_{1,2,3} \arrow X_{1,2,2}$ of the inequalities from Theorem \ref{Eff 122} cutting out the cone $\Eff(X_{1,2,2})$;
  \item the base locus inequalities corresponding to the fixed divisors $E_i,H_1-E_i, H_2-E_i-E_j$ from Lemmas \ref{lemma-baselocusexceptionals} and \ref{bsh2-ei-ej}.
    \end{itemize}
 The output is the cone of divisors cut out by the  inequalities below right, and whose extremal rays,  computed using the Normaliz files  \texttt{X123-ineqs}, are spanned by the divisors in the table below left:

\begin{minipage}[t]{0.4\textwidth}
  \begin{align*}
\rowcolors{2}{white}{gray!15}
\begin{array}{ccccc}
H_1 & H_2 & E_1 & E_2 & E_3 \\
\hline\hline
%\hhline{=====}
0 & 1 & 0 & 0 & 0 \\
0 & 1 & -1 & 0 & 0 \\
0 & 2 & -1 & -1 & -1 \\
1 & 0 & 0 & 0 & 0 \\
1 & 1 & -1 & -1 & 0 \\
1 & 1 & -1 & -1 & -1 \\
\end{array}
  \end{align*}
  \end{minipage}\quad
  \begin{minipage}[t]{0.4\textwidth}
    \begin{align*}
\rowcolors{2}{white}{gray!15}
\begin{array}{ccccc}
d_1 & d_2 & m_1 & m_2 & m_3 \\
\hline\hline
%\hhline{=====}
1 & 0 & 0 & 0 & 0 \\
0 & 1 & 0 & 0 & 0 \\
1 & 1 & 1 & 0 & 0 \\
1 & 2 & 1 & 1 & 0 \\
1 & 2 & 1 & 1 & 1 \\
0 & 0 & -1 & 0 & 0 \\
0 & 1 & 1 & 0 & 0 \\
1 & 1 & 1 & 1 & 0 \\
\end{array} 
\end{align*}
\ 
\end{minipage}

The irreducible effective classes excluded by these inequalities are precisely $E_i$ and $H_1-E_i$ and $H_2-E_i-E_j$, for all choices of $i,j$ with $i\ne j$.
Following \textbf{Step 3}, we add these classes to the  cone obtained above and we compute the extremal rays of the resulting cone: the files \texttt{X123-gens} compute the list of extremal rays and inequalities displayed in the statement of the theorem. All of them are represented by effective classes and this allows us to conclude that this is the effective cone $\Eff(X_{1,2,3})$ as claimed.
\end{proof}

We notice that the $X_{1,2,3}$ behaves like the toric cases, namely there is no new generator of the effective cone that did not already appear on $X_{1,2,2}$.

\begin{remark} The effective cone of $X_{1,n,n+1}$ follows as an application of work of Hausen and S{\"u}{\ss} \cite{HS10}, where the authors proved more general results on the Cox rings of algebraic varieties with torus actions. In particular, the action of $(\CC^\ast)^n$ on $\PP^n$ extends to a complexity-1 action on $X_{1,n,n+1}$, and \cite[Theorem 1.3]{HS10} gives an explicit description of the generators of the Cox ring, and hence the effective cone, in terms of this action.
\end{remark}

\subsection*{4 points}

We will follow the procedure outlined in Section \ref{effconemethod} in order to describe the cone of effective divisors of $X_{1,2,4}$. 

\begin{theorem}\label{Eff 124}
The effective cone $\Eff(X_{1,2,4})$ is given by the list of generators below left and inequalities below right:

\begin{minipage}[t]{0.4\textwidth}
  \begin{align*}
\rowcolors{2}{white}{gray!15}
\begin{array}{cccccc}
H_1 & H_2 & E_1 & E_2 & E_3 & E_4\\
\hline\hline
%\hhline{=====}
0 & 0 & 1 & 0 & 0 & 0\\
1 & 0 & -1 & 0 & 0 & 0\\
0 & 1 & -1 & -1 & 0 & 0\\
1 & 1 & -1 & -1 & -1 & -1\\
\end{array}
  \end{align*}
  \end{minipage}\quad
  \begin{minipage}[t]{0.4\textwidth}
    \begin{align*}
    \rowcolors{2}{white}{gray!15}
\begin{array}{cccccc}
d_1 & d_2 & m_1 & m_2 & m_3 & m_4\\
\hline\hline
%\hhline{=====}
1 & 0 & 0 & 0 & 0 & 0\\
0 & 1 & 0 & 0 & 0 & 0\\
1 & 1 & 1 & 0 & 0 & 0\\
1 & 2 & 1 & 1 & 0 & 0\\
1 & 2 & 1 & 1 & 1 & 0\\
1 & 3 & 1 & 1 & 1 & 1\\
2 & 2 & 1 & 1 & 1 & 1\\
2 & 3 & 2 & 1 & 1 & 1\\
\end{array}
\end{align*}

\end{minipage}
\end{theorem}
\begin{proof}
We apply the method of Section \ref{effconemethod} with the following inputs:
  \begin{itemize}
  \item the pullback via the blowdown morphism $X_{1,2,4} \arrow X_{1,2,3}$ of the inequalities from Theorem \ref{Eff 123} cutting out the cone $\Eff(X_{1,2,3})$;
  \item the base locus inequalities corresponding to the fixed divisors $E_i,H_1-E_i, H_2-E_i-E_j$ from Lemmas \ref{lemma-baselocusexceptionals} and \ref{bsh2-ei-ej}.
    \end{itemize}
A minimal set of generators of the so obtained cone is computed using the Normaliz files \texttt{X124-ineqs}.    
The irreducible effective classes excluded by these inequalities are precisely $E_i$ and $H_1-E_i$ and $H_2-E_i-E_j$.
We now add these classes to the  cone obtained above and we compute the extremal rays of the resulting cone using the files \texttt{X124-gens}. As a result, we obtain the lists of extremal rays and inequalities displayed in the statement of the theorem. 
Observe that the divisor $H_1+H_2-\sum_{i=1}^4E_i$ is effective: this follows from Lemma \ref{vdim lower bound} since its virtual dimension is $2 \cdot 3 - 4 -1 =1$. Therefore all extremal rays are represented by effective classes, proving that this is the effective cone of $X_{1,2,4}$.
\end{proof}

This is the first interesting case that does not behave in a toric manner. Indeed we see that the new generator $H_1+H_2-\sum_{i=1}^4E_i$, that is not inherited from $X_{1,2,s}$, with $s\le 3$, appears.

\subsection*{5 points}

We will follow the procedure outlined in Section \ref{effconemethod} in order to describe the cone of effective divisors of $X_{1,2,5}$. We start by proving a general statement which we will use now in the case $n=2$ and, in Section \ref{section-dim4}, in the case $n=3$.

\begin{lemma}\label{lemma-divisor1} A smooth divisor of bidegree $(1,1)$ in $\PP^1 \times \PP^n$ is isomorphic to the blowup of $\PP^n$ in a linear space of codimension 2.

 More generally, given a set of $s$  points in general position in $\PP^1 \times \PP^n$, let $V_s$ be a divisor of bidegree $(1,1)$ passing through all the points, and $\widetilde{V}_s$ its proper transform on $X_{1,n,s}$. Then $\widetilde{V}_s$ is isomorphic to the blowup of $\PP^n$ in a linear space of codimension 2 and $s$ points. 
\end{lemma}
\begin{proof} The proper transform $\widetilde{V}_s$ is isomorphic to the blowup of $V_s$ in $s$ points, so the second statement follows directly from the first.
  
  To prove the first statement, let $V$ be a smooth divisor of bidegree $(1,1)$ on $\PP^1 \times \PP^n$. We can choose coordinates $[u,v]$ on $\PP^1$ so that $V$ is defined by a bihomogeneous equation
  \begin{align*}
    u L_1 + v L_2 &= 0,
  \end{align*}
  where and $L_1$ and $L_2$ are independent linear forms in $x_0,\ldots,x_n$. The projection map $\pi \colon V \arrow \PP^n$ is an isomorphism outside the locus
  \begin{align*}
    \left\{ [x_0,\ldots,x_n] \in \PP^n \mid L_1(x_0,\ldots,x_n) = L_2(x_0,\ldots,x_n) =0 \right\},
  \end{align*}
  which is a codimension 2 linear space $\Lambda \subset \PP^n$. Over each point of $\Lambda$ the fibre of $\pi$ is $\PP^1$, so $\pi^{-1}(\Lambda)$ is a (Cartier) divisor in $V$, and therefore $\pi$ factors through the blowup $\operatorname{Bl}_\Lambda(\PP^n) \arrow \PP^n$. Since both varieties are smooth of Picard number 2, Zariski's Main Theorem shows that $V \arrow \operatorname{Bl}_\Lambda(\PP^n)$ is an isomorphism.
\end{proof}

\begin{proposition}\label{new generators 5 points}
The divisors $H_1+H_2-\sum_{i=1}^5E_i$ and $2H_2-\sum_{i=1}^5E_i$ are effective and fixed on $X_{1,2,5}$ and they are both isomorphic to a degree $3$ del Pezzo surface of type $X_{1,1,5}$. 
\end{proposition}
\begin{proof}
Both divisors have virtual dimension $0$, hence the first statement holds by Lemma \ref{vdim lower bound}.
One can show that the divisors are fixed by direct computation: it is indeed enough to show that $5$ randomly chosen points impose $5$ independent conditions to the linear system of surfaces in $\PP^1\times\PP^2$ of bidegree $(1,1)$ and $(0,2)$ respectively.

For the divisor $H_1+H_2-\sum_{i=1}^5 E_i$, the second statement follows from Lemma \ref{lemma-divisor1}. The divisor $2H_2-\sum_{i=1}^5E_i$ is the proper transform of a surface $S \subset \PP^1 \times \PP^2$ which is the preimage of a smooth conic in $\PP^2$. Since $S$ is isomorphic to $\PP^1 \times \PP^1$ the result follows. 
\end{proof}

We are now ready to prove our main statement.

\begin{theorem}\label{Eff 125}
The effective cone $\Eff(X_{1,2,5})$ is given by the list of generators below left and inequalities below right:

\begin{minipage}[t]{0.4\textwidth}
  \begin{align*}
\rowcolors{2}{white}{gray!15}
\begin{array}{ccccccc}
H_1 & H_2 & E_1 & E_2 & E_3 & E_4 & E_5\\
\hline\hline
%\hhline{=====}
0 & 0 & 1 & 0 & 0 & 0 & 0\\
0 & 1 & -1 & -1 & 0 & 0 & 0\\
1 & 0 & -1 & 0 & 0 & 0 & 0\\
1 & 1 & -1 & -1 & -1 & -1 & -1\\
0 & 2 & -1 & -1 & -1 & -1 & -1\\
\end{array}
  \end{align*}
  \end{minipage}\quad
  \begin{minipage}[t]{0.4\textwidth}
    \begin{align*}
    \rowcolors{2}{white}{gray!15}
\begin{array}{ccccccc}
d_1 & d_2 & m_1 & m_2 & m_3 & m_4 & m_5\\
\hline\hline
%\hhline{=======}
1 & 0 & 0 & 0 & 0 & 0 & 0\\
0 & 1 & 0 & 0 & 0 & 0 & 0\\
1 & 1 & 1 & 0 & 0 & 0 & 0\\
1 & 2 & 1 & 1 & 0 & 0 & 0\\
1 & 2 & 1 & 1 & 1 & 0 & 0\\
1 & 3 & 1 & 1 & 1 & 1 & 0\\
1 & 4 & 1 & 1 & 1 & 1 & 1\\
2 & 2 & 1 & 1 & 1 & 1 & 0\\
2 & 3 & 2 & 1 & 1 & 1 & 0\\
3 & 3 & 2 & 1 & 1 & 1 & 1\\
3 & 4 & 3 & 1 & 1 & 1 & 1\\
\end{array}
\end{align*}

\end{minipage}
\end{theorem}
\begin{proof}

We apply the method of Section \ref{effconemethod} with the following inputs:
  \begin{itemize}
  \item the pullback via the morphism $X_{1,2,5} \arrow X_{1,2,4}$ of the inequalities from Theorem \ref{Eff 124} cutting out the cone $\Eff(X_{1,2,4})$;
  \item the base locus inequalities corresponding to the fixed divisors $E_i,H_1-E_i, H_2-E_i-E_j$ from Lemmas \ref{lemma-baselocusexceptionals} and \ref{bsh2-ei-ej},
    \end{itemize}
    using the Normaliz files  \texttt{X125-ineqs}.
 We then add the span of the classes $E_i,H_1-E_i, H_2-E_i-E_j$ to the so obtained cone,
using  the file \texttt{X125-gens}: we obtain the lists of extremal rays and inequalities displayed in the statement of the theorem. 

To conclude, we observe that the divisor $H_1+H_2-\sum_{i=1}^5 E_i$, and $2H_2-\sum_{i=1}^5 E_i$ are effective by Proposition \ref{new generators 5 points}, and that all other divisors are effective by Theorem \ref{Eff 124}. Therefore we can conclude that this is the list of extremal rays of the effective cone of $X_{1,2,5}$.
\end{proof}

\subsubsection*{Base locus lemmata for degree two divisors}\label{base locus degree 2}

Although we have already computed $\Eff(X_{1,2,5})$, there are additional base locus lemmas for divisors on this space that will be essential in computing the next case, namely $\Eff(X_{1,2,6})$. In this subsection we prove these base locus lemmas, for the fixed divisors $H_1+H_2-\sum_{j=1}^5E_{i_j}$ and $2H_2-\sum_{j=1}^5E_{i_j}$ on $X_{1,2,s}$, for $s\ge 5$. %({\color{blue}For ease of notation, we state each of the base locus lemmas \ref{baselocus 1,1}--\ref{baselocus 1,4} for a single divisor class, but the analogous statements apply equally to any divisor obtained by permutation of indices.})

%% {\color{blue}
%% \marginpar{\tiny I had to split the statement considering separately the case with 5 points and the case with 6 or more}
\begin{lemma}\label{baselocus 1,1}
Fix $s\ge 5$ and consider the threefold $X_{1,2,s}$.
For $s=5$, the divisor $H_1+H_2-\sum_{i=1}^5E_{i}$ is contained in the base locus of $D=d_1H_1+d_2H_2-\sum_{i=1}^5m_iE_i$ with multiplicity at least $$\max\left\{0,\sum_{j=1}^5m_{i_j}-d_1-3d_2, 
\left\lceil \sum_{j=1}^5m_{i_j}-\frac{3d_1+5d_2}{2}\right\rceil \right\}.$$
For $s>5$, the divisor $H_1+H_2-\sum_{j=1}^5E_{i_j}$ is contained in the base locus of $D=d_1H_1+d_2H_2-\sum_{i=1}^sm_iE_i$, for $1\le i_1<\cdots<i_5<i_6\le s$, with multiplicity at least $$\max\left\{0,\sum_{j=1}^5m_{i_j}-d_1-3d_2, 
\left\lceil \sum_{j=1}^5m_{i_j}-\frac{3d_1+5d_2}{2} \right\rceil,
\left\lceil \frac{2\sum_{j=1}^5m_{i_j}+m_{i_6}-4d_1+5d_2}{3}\right\rceil \right\}.$$

\end{lemma}
%% }

\begin{proof} 
%Under the assumption that $m_i\ge 0$ for every $i=1,\dots,s$, 

We first prove the statement for
%it is enough to prove the statement for 
divisors on $X_{1,2,5}$, namely for $D=d_1H_1+d_2H_2-\sum_{i=1}^5m_iE_i$.
%We will assume that $\sum_{i=1}^5m_i-d_1-3d_2>0$, otherwise the claim is trivial.

It follows from  Proposition \ref{new generators 5 points} that $\tilde{S}_1=H_1+H_2-\sum_{i=1}^5E_i$ is a degree $3$ del Pezzo surface, the blowup of $\PP^2$ in $6$  points in general position, with Picard group generated by the class of a line $h$ and the class of $6$ exceptional divisors $e,e_i=E_i|_{\tilde{S}_1}$, $i=1,\dots,5$. 
Now, for $i=1,2$, write $({H_i}|_{\tilde{S}_1})=a_ih-b_ie$; since both are effective and irreducible, we can assume $a_i\ge b_i\ge 0$. The following equations determine the values of $a_1,b_1,a_2,b_2$:
$0={H_1}^2|_{\tilde{S}_1}=a_1^2-b_1^2$, $1={H_2}^2|_{\tilde{S}_1}=a_2^2-b_2^2$, $0={H_1H_2}|_{\tilde{S}_1}=a_1a_2-b_1b_2$. This yields:
\begin{align*}
{H_1}|_{\tilde{S}_1}& =h-e,\\ 
{H_2}|_{\tilde{S}_1}&=h.
\end{align*}
The restriction is $D|_{\tilde{S}_1}=(d_1+d_2)h-d_1e-\sum_{i=1}^5m_ie_i$. 
In order to prove the statement, we consider two moving curve classes on $\tilde{S}_1$, and for each of them we apply Lemma \ref{lemma-baselocusgeneral} to obtain the claimed base locus multiplicity.
\begin{itemize}
\item
Consider first of all class $C_1=3h-2e-\sum_{i=1}^5e_i$, whose irreducible representatives sweep out $\tilde{S}_1$.
Notice  that 
\begin{align*}
\tilde{S}_1\cdot C_1 &= \left(H_1+H_2-\sum_{i=1}^5E_i\right)\cdot C_1= \left((h-e)+h-\sum_{i=1}^5e_i\right)\cdot C_1=-1\\
D|_{\tilde{S}_1}\cdot C_1& =d_1+3d_2-\sum_{i=1}^5m_i,
\end{align*}
so we can conclude by applying Lemma \ref{lemma-baselocusgeneral} that the integer $\sum_{j=1}^5m_{i_j}-d_1-3d_2$ is a lower bound for the multiplicity of containment of $\tilde{S}_1$ in the base locus of $D$.
\item
 %% {\color{blue}
Secondly, consider the curve class $C_2=5h-2e-2\sum_{i=1}^5e_i$, whose irreducible representatives sweep out $\tilde{S}_1$. We conclude observing that 
\begin{align*}
\tilde{S}_1\cdot C_2 &= \left((h-e)+h-\sum_{i=1}^5e_i\right)\cdot C_2=-2\\
D|_{\tilde{S}_1}\cdot C_2& =3d_1+5d_2-2\sum_{i=1}^5m_i.
\end{align*}
and applying Lemma \ref{lemma-baselocusgeneral}. we obtain that $\left\lceil \frac{2\sum_{i=1}^5m_i-3d_1-5d_2}{2} \right\rceil$ is a lower  bound for the multiplicity of containment of $\tilde{S}_1$ in the base locus of $D$.
%% }
\end{itemize}
%% {\color{blue}
This completes the proof of the first statement.

Under the assumption that $m_i\ge 0$ for every $i=1,\dots,s$, it is enough to prove the second statement for 
divisors on $X_{1,2,6}$. Without loss of generality we may assume that the fixed bidegree $(1,1)$ surface is based at the first five points: $\tilde{S}_1=H_1+H_2-\sum_{i=1}^5E_i$. Consider the fixed fibral curve $l_1-e_6$, that intersects $\tilde{S}_1$ transversally in a point $q$. Since the points $p_1,\dots,p_6$ are in general position in $\mathbb{P}^1\times\mathbb{P}^2$, then so are the points $p_1,\dots,p_5,q$ in the del Pezzo surface ${S}_1$ of class $H_1+H_2$, whose blow-up is $\tilde{S}_1$. Let us consider the blow-up $\tilde{X}_{1,2,6}$  of $X_{1,2,6}$ at $q$ with exceptional divisor $E_q$, and the induced blow-up of $\tilde{S}_1$, that we will denote by $\tilde{\tilde{S}}_1$. The latter is a degree-$2$ del Pezzo surface, whose Picard group is generated by $e_q=E_q|_{\tilde{\tilde{S}}_1}$ and by $h,e_1,\dots,e_5$ that, abusing notation, are the pull-backs of the corresponding classes on $\tilde{S}_1$. 

Now, consider a divisor $D$ with class $d_1H_1+d_2H_2-\sum_{i=1}^6m_iE_i$ on $X_{1,2,6}$. If $P_1$ and $P_2$ are general divisors of class $H_2-E_6$, then we compute
\begin{align*}
  (D_{|P_1}) \cdot ((P_2)_{|P_1}) &= D \cdot P_1 \cdot P_2 \\
  &=D \cdot (l_1-e_6)\\
  &=d_1-m_6.
\end{align*}
This means that if $m_6>d_1$, then the curve $(P_2)_{|P_1}=l_1-e_6$ is contained in $D_{|P_1}$ at least $m_6-d_1$ times. Therefore $D_{|P_1}$ has multiplicity at least $m_6-d_1$ at every point of  $l_1-e_6$, and hence so too does $D$. So if $\tilde{D}$ is the strict transform of $D$, then 
$$\tilde{D}|_{\tilde{\tilde{S}}_1}=(d_1+d_2)h-d_1e-\sum_{i=1}^5m_ie_i-m_qe_q,$$ with $m_q\ge m_6-d_1$. 
We will show that $\left\lceil \frac{2\sum_{i=1}^5m_i+m_6-4d_1-5d_2}{3}\right\rceil$ is a lower bound for the multiplicity of containment of $\tilde{\tilde{S}}_1$ in the base locus of $\tilde{D}$.
In order to do so, we consider the curve class $C_3=5h-2e-2\sum_{i=1}^5e_i-e_q$, whose irreducible representatives sweep out $\tilde{\tilde{S}}_1$. 
We have
\begin{align*}
\tilde{\tilde{S}}_1\cdot C_3 &= \left((h-e)+h-\sum_{i=1}^5e_i-e_q\right)\cdot C_3=-3\\
\tilde{D}|_{\tilde{\tilde{S}}_1}\cdot C_3& \le3d_1+5d_2-2\sum_{i=1}^5m_i-(m_6-d_1).
\end{align*}
Using Lemma \ref{lemma-baselocusgeneral} we obtain that $\tilde{\tilde{S}}_1$ is in the base locus of $\tilde{D}$, and hence ${\tilde{S}}_1$ is in the base locus of ${D}$, at least $\left \lceil \frac{2\sum_{i=1}^5m_i-m_6-4d_1-5d_2}{3} \right\rceil$ times for divisors with $m_6>d_1$.

In order to conclude, we need to consider the case $m_6\le d_1$. We claim that under this assumption the following holds:
$$
\frac{2\sum_{i=1}^5m_i+m_6-4d_1-5d_2}{3}
\le 
\max\left\{0,\sum_{j=1}^5m_{i_j}-d_1-3d_2, \frac{2
\sum_{i=1}^5m_{i}-3d_1-5d_2}{2}\right\}.
$$
In order to prove the claim, we consider two cases.
First of all, let $2\sum_{j=1}^5m_{i_j}-3d_1-5d_2\le 0$. In this case 
$$
2\sum_{i=1}^5m_i+m_6-4d_1-5d_2=
\left(2\sum_{j=1}^5m_{i_j}-3d_1-5d_2\right)+(m_6-d_1)\le0
$$
and the claim holds.
Otherwise, if $2\sum_{j=1}^5m_{i_j}-3d_1-5d_2\ge0$, then
\begin{align*}
&3\left(2\sum_{j=1}^5m_{i_j}-3d_1-5d_2\right)
-2\left(2\sum_{i=1}^5m_{i_j}+m_6-4d_1-5d_2\right)\\
=&\left(2\sum_{j=1}^5m_{i_j}-3d_1-5d_2\right)+2(d_1-m_6)\\ \ge&0,
\end{align*}
and the claim holds in this case too.
%% }

 \end{proof}

\begin{lemma}\label{baselocus 0,2}
The divisor $2H_2-\sum_{j=1}^5E_{i_j}$ is contained in the base locus of $D=d_1H_1+d_2H_2-\sum_{i=1}^sm_iE_i$, for $1\le i_1<\cdots<i_5\le s$, with multiplicity at least $\max\{0,\sum_{j=1}^5m_{i_j}- 3d_1-2d_2\}$.
\end{lemma}
\begin{proof}
Under the assumption that $m_i\ge 0$ for every $i=1,\dots,s$, it is enough to prove the statement for divisors on $X_{1,2,5}$, namely for $D=d_1H_1+d_2H_2-\sum_{i=1}^5m_iE_i$.
We will assume that $\sum_{i=1}^5m_i-3d_1-2d_2>0$, otherwise the claim is trivial.

It follows from  Proposition \ref{new generators 5 points} that $\tilde{S}_2=2H_2-\sum_{i=1}^5E_i$ is a $X_{1,1,5}$, with Picard group generated by the two rulings $h_1,h_2$ and by $5$ exceptional curves $e_i=E_i|_{\tilde{S}_2}$, $i=1,\dots,5$. 
Now, for $i=1,2$, write ${H_i}|_{\tilde{S}_2}=a_ih_1+b_ih_2$; with $a_i, b_i\ge 0$. Since ${H_i}^2|_{S_2}=0$, for $i=1,2$, and ${H_1H_2}|_{S_2}=2$, we obtain that one of the following sets of relations holds: $a_ib_j=2, a_j=b_i=0$ for either $(i,j)=(1,2)$ or $(i,j)=(2,1)$. Without loss of generality, we may choose the first set.  From the equality  $(2H_1+H_2)|_{\tilde{S}_2}=2h_1+2h_2$, cf. the proof of Proposition \ref{new generators 5 points}, we obtain that $a_1=1,b_2=2$. This yields:
\begin{align*}
{H_1}|_{\tilde{S}_2}& =h_1,\\ 
{H_2}|_{\tilde{S}_2}&=2h_2.
\end{align*}
The restriction is $D|_{\tilde{S}_2}=d_1h_1+2d_2h_2-\sum_{i=1}^5m_ie_i$. Consider the moving curve class $C=h_1+3h_2-\sum_{i=1}^5e_i$, whose irreducible representatives sweep out $\tilde{S}_2$ (cf. Proposition \ref{proposition-x115}).
Since
\begin{align*}
\tilde{S}_2\cdot C &= \left(2H_2-\sum_{i=1}^5E_i\right)\cdot C= \left(4h_2-\sum_{i=1}^5e_i\right)\cdot C=-1\\
D|_{\tilde{S}_2}\cdot C& =3d_1+2d_2-\sum_{i=1}^5m_i,
\end{align*}
we can conclude applying Lemma \ref{lemma-baselocusgeneral}.
\end{proof}

%%%%%%%%%%

\subsection*{6 points}
%% {\color{blue}
In this subsection we will compute our final example in dimension 3, the effective cone of $X_{1,2,6}$. To do so, we first identify a final set of fixed divisors and give the associated base locus inequality. 

\begin{lemma}\label{baselocus 1,4}
  The  divisor class $H_1+4H_2-3E_1-2\sum_{i=2}^6E_i-E_j$ is effective and fixed. The unique effective divisor $F$ in this class is irreducible. It is contained in the base locus of $D=d_1H_1+d_2H_2-\sum_{i=1}^6m_iE_i$ with multiplicity at least $\max\{0, 2m_1+\sum_{i=2}^6m_i-3d_1-3d_2\}$.
\end{lemma}
\begin{proof}
 Consider the divisor class $F=H_1+4H_2-3E_1-2\sum_{i=2}^6E_i$. The formula for virtual dimension in Defintion \ref{virtual dimension} gives $\vdim|F|=-1$, suggesting that $F$ has no global sections.

 However, consider the linear system $|G|$ corresponding to the divisor class $G=H_1+4H_2-3E_1$. We claim that $\dim|G|>\vdim|G|$, and therefore $\dim|F|>\vdim|F|$ also, so that we can conclude that $\dim|F|\ge0$. In order to prove the claim, we assign coordinates $x_0,x_1;y_0,y_1,y_2$ to $\PP^1;\PP^2$ and we consider  the generic degree $(1,4)$ form 
\begin{align*}
f(x_0,x_1,y_0,y_1,y_2)=&\ x_0(a_{0000}y_0^4+a_{0001}y_0^3y_1+a_{0002}y_0^3y_2+\cdots+a_{2222}y_2^4)+\\
& \ x_1(b_{0000}y_0^4+b_{0001}y_0^3y_1+b_{0002}y_0^3y_2+\cdots+b_{2222}y_2^4).
\end{align*}
Imposing a triple point at  $p_1=([1:0],[1:0:0])$ corresponds to imposing that the second order derivatives of $f$ vanish when evaluated at $(1,0,1,0,0)$. However, since our form has degree 1 in the variables $x_0,x_1$, the second derivative $\partial^2 f/\partial x_1^2$ is automatically zero, and so the triple point at $p_1$ imposes 9 conditions rather than 10 as in the formula of Definition \ref{virtual dimension}. Therefore $\dim|G|>\vdim |G|$ as claimed. %% It is an easy calculation to check that this gives nine independent conditions (and not ten as expected), because some derivatices are identically zero and others give redundant conditions, as follows:

%%  \begin{align*}
%% \frac{\partial}{\partial x_0y_0}f(1,0,1,0,0)&=4a_{0000}=0,\\
%% \frac{\partial}{\partial x_0y_1}f(1,0,1,0,0)&=a_{0001}=0,\\
%% \frac{\partial}{\partial x_0y_2}f(1,0,1,0,0)&=a_{0002}=0,\\
%% \frac{\partial}{\partial x_1y_0}f(1,0,1,0,0)&=4b_{0000}=0,\\
%% \frac{\partial}{\partial x_1y_1}f(1,0,1,0,0)&=b_{0001}=0,\\
%% \frac{\partial}{\partial x_1y_2}f(1,0,1,0,0)&=b_{0002}=0,\\
%% \frac{\partial}{\partial y_1^2}f(1,0,1,0,0)&=2a_{0011}=0,\\
%% \frac{\partial}{\partial y_1y_2}f(1,0,1,0,0)&=a_{0012}=0,\\
%% \frac{\partial}{\partial y_2^2}f(1,0,1,0,0)&=2a_{0022}=0,
%% \end{align*}
%%  \begin{align*}
%% \frac{\partial}{\partial y_0^2}f(1,0,1,0,0)&=12a_{0000}=0 \textrm{ (redundant)},\\
%% \frac{\partial}{\partial y_0y_1}f(1,0,1,0,0)&=3a_{0001}=0 \textrm{ (redundant)},\\
%% \frac{\partial}{\partial y_0y_2}f(1,0,1,0,0)&=3a_{0002}=0 \textrm{ (redundant)},\\
%% \frac{\partial}{\partial x_0^2}f&\equiv0,\\
%%  \frac{\partial}{\partial x_0x_1}f&\equiv0,\\ 
%%  \frac{\partial}{\partial x_1^2}f&\equiv0,
%%  \end{align*}

Next we claim that any effective divisor $F$ in the class $H_1+4H_2-3E_1-\sum_{i=2}^6 E_i$ is irreducible. To see this, note that since $F \cdot l_1=1$, the divisor $F$ contains a unique irreducible component which dominates $\PP^2$. Write $F=F_1+F_2$ where $F_1$ is the unique irreducible component which dominates $\PP^2$, and $F_2$ is a sum of irreducible components contracted by $\pi_2 \colon X \arrow \PP^2$. The class of $F_2$ must be a sum of classes of the form $E_i$, $H_2-E_i-E_j$, and $2H_2-\sum_{j=1}^5 E_{i_j}$. Assuming $F_2$ is nonzero, it is then easy to find a curve class $C$ of the form $3l_1+5l_2-\sum_{j=1}^5 E_{i_j}$ such that $F_2 \cdot C>F \cdot C$, and hence $F_1 \cdot C<0$. By Lemma \ref{baselocus 1,1} this implies that $F_1$ is a divisor of the form $H_1+H_2-\sum_{k=1}^5 E_{i_k}$, but for any such divisor we find that $F-F_1$ has the form $3H_2-2E_{i_1}-2E_{i_2}-E_{i_3}-\cdots-E_{i_6}$. Since there is no cubic curve in $\PP^2$ with 2 double points and passing through 4 other general points, $F-F_1$ is not effective. So we must have that $F_2=0$, hence $F$ is irreducible as required.

%% We compute:
%% \begin{align*}
%%  F \cdot \left( 3l_1+5l_2 - 2\sum_{i \in I} E_i \right) &= 1\\
%%  E_j \cdot \left( 3l_1+5l_2 - 2\sum_{i \in I} E_i \right) &= 2\quad \text{ for }  j \in I\\
%%  \left(H_2-E_i-E_j\right) \cdot \left( 3l_1+5l_2 - 2\sum_{k \in K} E_k \right) &= 3 \quad \text{ for } j \notin K\\
%%  \left(2H_2-\sum_{j \in J} E_j\right) \cdot \left( 3l_1+5l_2 - 2\sum_{k \in K} E_k \right)&=2 \quad \text{ for } J \neq K 
%% \end{align*}
%% where all the index sets $I, \, J, \, K$ contain 5 elements. In particular we see if $F_2$ is nonempty, then $F_1 \cdot C= (F-F_2) \cdot C<0$ for some curve class $C$ of the form $3l_1+5l_2-2\sum_{i \in I}$, and therefore by Lemma \ref{baselocus 1,1} we must have $F_1=H_1+H_2-\sum_{i \in I} E_i$. But
%% \begin{align*}
%% F-(H_1+H_2-\sum_{i=1}^5 E_i) &= 3H_2-2E_1-2E_6-E_2-E_3-E_4-E_5
%% \end{align*}
%% is not effective, since there is no plane cubic with two double points passing through 4 other general points (and the same conclusion holds for permutations of the indices). Therefore $F_2$ must be empty, showing that $F$ is irreducible as claimed. 

 To prove that $F$ is fixed and the claimed multiplicity of containment holds, we apply Lemma \ref{lemma-baselocusgeneral} with the curve class $C=3l_1+3l_2-2e_1-\sum_{i=2}^6 e_i$. We compute
 \begin{align*}
 \left(H_1+4H_2-3E_1 - 2\sum_i E_i \right) \cdot C &= 15 -6-10 =  -1. 
 \end{align*}
 Therefore it suffices to prove that irreducible curves of class $C$ sweep out a divisor on $X_{1,2,6}$, which must then be $F$.
 
 Let $\Gamma$ be an irreducible cubic curve in $\PP^2$ with a node at $\pi(p_1)$ and passing through the points $\pi(p_i)$ for $i=2,\ldots,6$. There is a 1-parameter family of such curves, sweeping out $\PP^2$. Now let
 \begin{align*}
   \nu \colon \PP^1 \times \PP^1 \arrow \PP^1 \times \Gamma \subset \PP^1 \subset \PP^2
 \end{align*}
 be the product of the identity map with the normalisation of $\Gamma$. Under this map, a curve of bidegree $(a,b)$ in $\PP^1 \times \PP^1$ maps to a curve of bidegree $(a,3b)$ in $\PP^1 \times \PP^2$. The preimages of the 6 points $p_1,\ldots,p_6$ give us 7 points $q_1,\tilde{q}_1,\ldots,q_6$ in $\PP^1 \times \PP^1$ such that $\nu(q_1)=\nu(\tilde{q}_1)=p_1$ and $\nu(q_i)=p_i$ for $i=2,\ldots,6$. Counting parameters, we see that there is a curve $C'$ of bidegree $(3,1)$ through these 7 points. Moreover, such a curve must be irreducible, since by generality a curve of bidegree $(a,0)$ can pass through at most $a$ points and a curve of bidegree $(0,1)$ can pass through at most 2 points $q_0$ and $\tilde{q}_0$. The image of $C'$ in $\PP^1 \times \PP^2$ is then an irreducible curve of bidegree $(3,3)$ through the 6 points $p_1,\ldots,p_6$, and since the map $\nu$ identifies the $\PP^1$-rulings through $q_0$ and $\tilde{q}_0$, the image $C$ has a node at $p_1$. The proper transform of $C'$ on $X_{1,2,6}$ is then an irreducible curve in the class $C$. By construction $\pi_2(C)=\Gamma$, so as we vary $\Gamma$ to sweep out $\PP^2$ these irreducible curves of class $C$ then sweep out a divisor on $X_{1,2,6}$. \end{proof}

\begin{remark}\label{Laface-Moraga0}
Laface and Moraga \cite{LM16} introduced the  \emph{fibre-expected dimension} for linear systems of divisors on blowups of $(\PP^1)^n$, providing a lower bound for the dimension that improves the one given by the virtual dimension. This takes into account that (the proper transforms after blowing up the points of) certain fibres of the natural projections $(\PP^1)^n\to (\PP^1)^r$ are  contained with multiplicity in the base locus of the divisors, whenever the multiplicities are large enough with  respect to the degrees. In algebraic terms, this corresponds to observing that certain partial derivatives of the multidegree polynomials corresponding to the linear systems vanish identically, therefore they should not give a contribution to the dimension count.

Their approach can be used verbatim in the case of blowups of $\PP^1\times\PP^n$, where each fixed line $C_i=l_1-e_i$ is contained at least $-D\cdot C_i$ times in the base locus of $D$ and, if $D\cdot C_i\ge2$,  this gives a contribution to the dimension count. 

In particular the argument  proposed in the proof of Lemma \ref{baselocus 1,4} to show that the divisor $H_1+4H_2-3E_1$ on $X_{1,2,s}$ has dimension strictly larger than expected is a generalisation of Laface's and Moraga's idea.
\end{remark}
%\marginpar{\tiny{Most of this remark is the old Remark \ref{LafaceMoraga1}}}
%}

To compute the effective cone of $X_{1,2,6}$, we require one more restriction on the classes of effective divisors. To obtain this, consider a movable divisor $M$ given as the proper transform of a general bidegree $(2,1)$ surface $S$ in $\PP^1\times\PP^2$ containing all six points. The following lemma  describes the geometry of $M$. 
  \begin{lemma}\label{lemma-bideg21}
    A smooth surface $S$ of bidegree $(2,1)$ in $\PP^1 \times \PP^2$ is isomorphic to $\PP^1 \times \PP^1$. 
  \end{lemma} 
  \begin{proof}
Using adjunction we see that $-K_S=2{H_2}|_{S}$, so $-K_S$ is 2-divisible in $\Pic(S)$, and we also find that $\left(-K_S \right)^2=8$. So if we can prove that $-K_S$ is ample, then $S$ is del Pezzo of degree 8 with $-K_S$ 2-divisible, hence it must be $\PP^1 \times \PP^1$.
   
To prove that $-K_S$ is ample, it is enough to prove that, for a general $S$ as in the statement the projection $S \arrow \PP^2$ is finite. 
    Let $[u,v]$ and $[x,y,z]$ be homogeneous coordinates on $\PP^1$ and $\PP^2$ respectively. Then $S$ is given by an equation
    \begin{align*}
F(u,v,x,y,z) = u^2\Lambda_1+uv\Lambda_2+v^2\Lambda_3 &=0,
    \end{align*}
    where the $\Lambda_i$ are linear forms in $x, \, y, \, z$. For a point $p \in \PP^2$, the fibre of $S \arrow \PP^2$ over $p$ is infinite if and only if $\Lambda_i(p)=0$ for $i=1, \, 2, \, 3$. We claim that if the surface $S$ is nonsingular, the set of such $p$ is empty, and therefore $S \arrow \PP^2$ is finite as required.

    To prove the claim, suppose for contradiction that the $\Lambda_i$ are linear forms in $x, \, y, \, z$ with a common zero $p \in \PP^2$. We will prove that $S$ must have a singular point.

    By changing coordinates in $\PP^2$ we can assume
    \begin{align*}
      \Lambda_1 = x, \quad \Lambda_2 &= x+y, \quad \Lambda_3 = y,
    \end{align*}
    so that $S$ is defined by the equation
    \begin{align*}
     F(u,v,x,y,z) = u^2x+v^2y+uv(x+y) &= 0.
    \end{align*}
    Note that $S$ contains the line
    \begin{align*}
          \left\{ \left([u,v],[0,0,1]\right) \mid [u,v] \in \PP^1 \right\}.
    \end{align*}
    Computing partial derivatives, we find
 \begin{align*}
   \frac{\partial F}{\partial u} = 2ux+v(x+y), &\quad \quad 
   \frac{\partial F}{\partial v} = 2vy + u(x+y),\\
    \frac{\partial F}{\partial x} = u^2+uv, \quad \quad
&    \frac{\partial F}{\partial y} = v^2+uv, \quad \quad
     \frac{\partial F}{\partial z} =0,
 \end{align*}
and therefore the surface $S$ is singular at the point $\left([1,-1],[0,0,1]\right)$, as required.
  \end{proof}

\begin{corollary} \label{mov(2,1)}
The divisor class $2H_1+H_2-\sum_{i=1}^6E_i$ is movable on $X_{1,2,6}$. Moreover the generic element $M$ of its linear system is isomorphic to $X_{1,1,6}$.
\end{corollary}
\begin{proof}
  By Lemma \ref{lemma-bideg21} we see that a general such $M$ is the the blowup of $\PP^1 \times \PP^1$ in 6 points, in other words it is isomorphic to $X_{1,1,6}$.
  
For every $j=1,\dots,6$, we have
\[
M=(H_1-E_j)+\left(H_1+H_2-\sum_{i=1}^6E_i+E_j\right).
\]
Since $M$ is linearly equivalent to a sum of fixed divisors in multiple ways, and there is no common summand in these decompositions, we can conclude that $M$ is movable.
\end{proof}

 %% \marginpar{\tiny{\color{blue} statement and proof to be changed}}
 
\begin{theorem}\label{Eff 126}
The effective cone $\Eff(X_{1,2,6})$ is given by the list of generators below left and inequalities below right:

\begin{minipage}[t]{0.4\textwidth}
  \begin{align*}
\rowcolors{2}{white}{gray!15}
\begin{array}{cccccccc}
H_1 & H_2 & E_1 & E_2 & E_3 & E_4 & E_5 & E_6\\
\hline\hline
%\hhline{=====}
0 & 0 & 1 & 0 & 0 & 0 & 0 & 0\\
0 & 1 & -1 & -1 & 0 & 0 & 0 & 0\\
1 & 0 & -1 & 0 & 0 & 0 & 0 & 0\\
1 & 1 & -1 & -1 & -1 & -1 & -1 & 0\\
0 & 2 & -1 & -1 & -1 & -1 & -1 & 0\\
1 & 4 & -3 & -2 & -2 & -2 & -2 & -2
\end{array}
  \end{align*}
  \end{minipage}\quad
  \begin{minipage}[t]{0.4\textwidth}
    \begin{align*}
    \rowcolors{2}{white}{gray!15}
\begin{array}{cccccccc}
d_1 & d_2 & m_1 & m_2 & m_3 & m_4 & m_5 & m_6\\
\hline\hline
%\hhline{=======}
0 & 1 & 0 & 0 & 0 & 0 & 0 & 0\\
1 & 0 & 0 & 0 & 0 & 0 & 0 & 0\\
1 & 1 & 1 & 0 & 0 & 0 & 0 & 0\\
1 & 2 & 1 & 1 & 0 & 0 & 0 & 0\\
1 & 2 & 1 & 1 & 1 & 0 & 0 & 0\\
1 & 3 & 1 & 1 & 1 & 1 & 0 & 0\\
1 & 4 & 1 & 1 & 1 & 1 & 1 & 0\\
1 & 4 & 1 & 1 & 1 & 1 & 1 & 1\\
2 & 2 & 1 & 1 & 1 & 1 & 0 & 0\\
2 & 3 & 2 & 1 & 1 & 1 & 0 & 0\\
3 & 3 & 2 & 1 & 1 & 1 & 1 & 0\\
%3 & 3 & 2 & 1 & 1 & 1 & 1 & 1\\
3 & 4 & 3 & 1 & 1 & 1 & 1 & 0\\
3 & 4 & 3 & 1 & 1 & 1 & 1 & 1\\
3 & 6 & 3 & 3 & 1 & 1 & 1 & 1\\
4 & 3 & 2 & 1 & 1 & 1 & 1 & 1\\
4 & 4 & 2 & 2 & 2 & 1 & 1 & 1\\
5 & 5 & 3 & 2 & 2 & 2 & 2 & 1\\
%5 & 5 & 2 & 2 & 2 & 2 & 2 & 2\\
6 & 5 & 2 & 2 & 2 & 2 & 2 & 2\\
6 & 6 & 4 & 2 & 2 & 2 & 2 & 1\\
7 & 6 & 3 & 3 & 2 & 2 & 2 & 2\\
7 & 7 & 3 & 3 & 3 & 3 & 2 & 2\\
7 & 8 & 3 & 3 & 3 & 3 & 3 & 3\\
9 & 10 & 5 & 5 & 3 & 3 & 3 & 3\\
10 & 10 & 4 & 4 & 4 & 4 & 4 & 3
\end{array}
\end{align*}
\end{minipage}
\end{theorem}

\begin{proof}
Let $M=2H_1+H_2-\sum_{i=1}^6E_i\cong X_{1,1,6}$ be the movable divisor of Corollary \ref{mov(2,1)}.
Let $h_1,h_2,e_2,\dots,e_6$ be the generators of the Picard group of $M$. 
The intersection table on $M$ is given by the following.
Set $H_1\vert_M=ah_1+bh_2$, for $a,b\ge 0$. From $(H_1H_2)|_M=H_1H_2(2H_1+H_2)=1$, we obtain $a+b=1$. From $H_1^2|_M=H_1H_2(2H_1+H_2)=0$, we obtain $2ab=0$. Therefore either $(a,b)=(1,0)$ or $(a,b)=(0,1)$. We obtain:
\begin{align*}
H_1\vert_M&=h_i,\\
H_2\vert_M&=h_1+h_2,\\
E_i\vert_M&=e_i, i=1,\dots,6,
\end{align*}
where the second equality is a consequence of the adjunction formula.
%% {\color{blue} We have a choice for which ruling $H_1\vert_M$ lies on, and we will consider both possibilities. If we call $M_1,M_2$  the corresponding surfaces, we obtain
%% \begin{align*}
%% D\vert_{M_1}&=(d_1+d_2)h_1+d_2h_2-\sum_{i=1}^6m_ie_i,\\
%% D\vert_{M_2}&=d_1h_1+(d_1+d_2)h_2-\sum_{i=1}^6m_ie_i.
%% \end{align*}}

Applying Proposition \ref{proposition-x116} to $D\vert_M$, we obtain the following necessary conditions for the effectivity of the divisors $D$ on $X_{1,2,6}$:

\begin{minipage}[t]{1\textwidth}
  \begin{align*}
\rowcolors{2}{white}{gray!15}
\begin{array}{cccccccc}
d_1 & d_2 & m_1 & m_2 & m_3 & m_4 & m_5 & m_6\\
\hline\hline
%\hhline{=====}
1 & 3 & 1 & 1 & 1 & 1 & 0 & 0\\
1 & 4 & 1 & 1 & 1 & 1 & 1 & 1\\
2 & 5 & 2 & 2 & 1 & 1 & 1 & 1\\
3 & 6 & 2 & 2 & 2 & 2 & 1 & 1\\
3 & 6 & 3 & 2 & 1 & 1 & 1 & 1\\
3 & 7 & 3 & 2 & 2 & 2 & 1 & 1\\
3 & 7 & 2 & 2 & 2 & 2 & 2 & 2\\
3 & 8 & 3 & 2 & 2 & 2 & 2 & 2\\
4 & 8 & 3 & 3 & 2 & 2 & 2 & 1\\
4 & 9 & 3 & 3 & 3 & 2 & 2 & 2\\
5 & 10 & 3 & 3 & 3 & 3 & 3 & 2\\
\end{array}
\end{align*}
\
\end{minipage}

%% {\color{blue} Applying the same proposition to $D\vert_{M_2}$ we obtain a similar list of inequalities, that is like the above with the roles of $d_1$ and $d_2$ exchanged.}

We now apply the method of Section \ref{effconemethod} with the following inputs:
  \begin{itemize}
     \item the above inequalities coming from the effectivity of $D|_M$;
      \item the pullback via the morphism $X_{1,2,6} \arrow X_{1,2,5}$ of the inequalities from Theorem \ref{Eff 125} cutting out the cone $\Eff(X_{1,2,5})$;
\item the base locus inequalities corresponding to the fixed divisors $E_i$, $H_1-E_i$, $H_2-E_i-E_j$ from Lemmas \ref{lemma-baselocusexceptionals} and \ref{bsh2-ei-ej};
\item the base locus inequalities corresponding to the fixed divisors
  $H_1+H_2-\sum_{j=1}^5E_{i_j}$ and $2H_2-\sum_{j=1}^5E_{i_j}$ from Lemmas \ref{baselocus 1,1} and \ref{baselocus 0,2}.
  \item the base locus inequalities corresponding to the fixed divisors $H_1+4H_2-2\sum_jE_j - E_i$ from Lemma \ref{baselocus 1,4}.
  \end{itemize}
The resulting cone is computed with the files  \texttt{X126-ineqs}. To the list of extremal rays, we add the rays spanned by the fixed divisors $E_i$, $H_1-E_i$, $H_2-E_i-E_j$, $H_1+H_2-\sum_{j=1}^5E_{i_j}$, $2H_2-\sum_{j=1}^5E_{i_j}$, and $H_1+4H_2-2\sum_{j=1}^6 E_j -E_i$, for all index permutations. A minimal set of generators and inqualities are obtained using the Normaliz files \texttt{X126-gens}. The former is given by the list of divisors on the left hand side of the statement of this theorem. Since they are all effective, they span the effective cone of $X_{1,2,6}$. 
\end{proof}

\section{Effective divisors on some threefolds}\label{section-dim3-extra}

In this section we compute the effective cones of certain threefolds obtained by blowing up either $\PP^3$ or $\PP^1 \times \PP^2$ in a line and a set of points. These results will be used as inputs in the computations of Section \ref{section-dim4}, when we use our method to determine the cones $\Eff(X_{1,3,s})$ for $s \leq 6$. 

We start with the following result, which is a special case of \cite[Theorem 5.1]{BDP16}. In the statement below $\mathcal E_i$ denotes the exceptional divisor of the blowup of a point in $\PP^3$. The Normaliz files {\tt X036-gens} show that the cone spanned by the generators in the left-hand table is indeed cut out by the inequalities in the right-hand table.

\begin{theorem}\label{theorem-x036}
  Let $X_{0,3,6}$ denote the blowup of $\PP^3$ in a set of 6  points in general position. The effective cone $\Eff(X_{0,3,6})$ is given  by the list of generators below left and inequalities below right:
    
\begin{minipage}[t]{0.4\textwidth}
  \begin{align*}
    \rowcolors{2}{white}{gray!15}
\begin{array}{ccccccc}
H & \mathcal E_1 & \mathcal E_2 & \mathcal E_3 & \mathcal E_4 & \mathcal E_5 & \mathcal E_6\\ 
\hline\hline
%\hhline{======}y
0 & 1 & 0 & 0 & 0 & 0 & 0\\
1 & -1 & -1 & -1 & 0 & 0 & 0\\
2 & -2 & -1 & -1 & -1 & -1 & -1
\end{array}
  \end{align*}
  \end{minipage}\quad
  \begin{minipage}[t]{0.4\textwidth}
    \begin{align*}
      \rowcolors{2}{white}{gray!15}
\begin{array}{ccccccc}
d & m_1 & m_2 & m_3 & m_4 & m_5 & m_6\\
\hline\hline
%\hhline{======}
1 & 0 & 0 & 0 & 0 & 0 & 0\\
1 & 1 & 0 & 0 & 0 & 0 & 0\\
3 & 1 & 1 & 1 & 1 & 0 & 0\\
3 & 1 & 1 & 1 & 1 & 1 & 0\\
5 & 2 & 2 & 1 & 1 & 1 & 1\\
7 & 2 & 2 & 2 & 2 & 2 & 2
\end{array}
    \end{align*}
  \end{minipage}
\end{theorem}
By blowing down one of the exceptional divisors we get the corresponding result for 5 points:
\begin{corollary}\label{corollary-x035}
  The effective cone $\Eff(X_{0,3,5})$ is given by  the list of generators below left and inequalities below right:

\begin{minipage}[t]{0.4\textwidth}
  \begin{align*}
    \rowcolors{2}{white}{gray!15}
\begin{array}{cccccc}
H & \mathcal E_1 & \mathcal E_2 & \mathcal E_3 & \mathcal E_4 & \mathcal E_5 \\ 
\hline\hline
%\hhline{======}
0 & 1 & 0 & 0 & 0 & 0\\
1 & -1 & -1 & -1 & 0 & 0\\
\end{array}
  \end{align*}
  \end{minipage}\quad
  \begin{minipage}[t]{0.4\textwidth}
    \begin{align*}
      \rowcolors{2}{white}{gray!15}
\begin{array}{cccccc}
d & m_1 & m_2 & m_3 & m_4 & m_5\\
\hline\hline
%\hhline{======}
1 & 0 & 0 & 0 & 0 & 0 \\
1 & 1 & 0 & 0 & 0 & 0 \\
3 & 1 & 1 & 1 & 1 & 0 \\
3 & 1 & 1 & 1 & 1 & 1
\end{array} 
    \end{align*}
  \end{minipage}
\end{corollary}
We will also need to consider certain threefolds which arise naturally as divisors in $\PP^1 \times \PP^3$; these threefolds can be described as blowups of $\PP^3$ in a line and a set of points. The following result helps in determining their effective cones.
\begin{proposition}\label{proposition-smallmodification}
Let $s \geq 0$. Let $Y_{L, \, s+1}$ denote the blowup of $\PP^3$ in a line and $s+1$ points in general position. Let $Z_{L, \, s}$ denote the blowup of $\PP^1 \times \PP^2$ in a line contained in a fibre of $\pi_1 \colon \PP^1 \times \PP^2 \arrow \PP^1$ and $s$ points in general postion. Then there is a small modification $\varphi \colon Y_{L,\, s+1} \dashrightarrow Z_{L,\,s}$. 
\end{proposition}
\begin{proof} 
  It is enough to consider the case $s=0$. For brevity we write $Y$, respectively $Z$, instead of $Y_{L,\, 1}$ and $Z_{L, \, 0}$. In this case both $Y$ and $Z$ are toric varieties, and we find the necessary $\varphi$ by considering their fans.

  Consider the fan of $\PP^3$ with ray generators $e_1, \, e_2, \, e_3, \, -\left(e_1+e_2+e_3 \right)$. Then $Y$ is obtained by star subdivision of the cones $\langle e_1, \, -\left(e_1+e_2+e_3 \right) \rangle$ corresponding to a line, and $\langle e_2, e_3, \, -\left(e_1+e_2+e_3 \right) \rangle$ corresponding to a point. The resulting fan $\Sigma$ has rays
  \begin{align*}
    \Sigma(1) &= \left\{ e_1, \, e_2, \, e_3, -\left(e_1+e_2+e_3 \right), \, -\left(e_2+e_3 \right), \, -e_1 \right\}.
  \end{align*}
  On the other hand, the fan of $\PP^1 \times \PP^2$ has ray generators $e_1, \, -e_1, \, e_2, \, e_3, \, -\left(e_2+e_3\right)$. This fan contains a cone $\langle -e_1, -\left(e_2+e_3\right) \rangle$ corresponding to a line contracted by $\pi_1 \colon \PP^1 \times \PP^2 \arrow \PP^1$. Blowing up along this line corresponds to star subdivision of the cone above, resulting in a fan $\Sigma^\prime$ with ray generators
  \begin{align*}
    \Sigma^\prime(1) &= \left\{e_1, \, -e_1, \, e_2, \, e_3, -\left(e_2+e_3 \right), -\left(e_1+e_2+e_3 \right) \right\} \\
    &= \Sigma(1).
  \end{align*}
  Since these two fans have the same rays, the evident rational map $\varphi \colon Y \dashrightarrow Z$ and its inverse do not contract any torus-invariant divisor on either side, and therefore $\varphi$ is a small modification as required.
\end{proof}
The small modification $\varphi$ allows us to identify divisors on $Y_{L,\, s+1}$ and $Z_{L,\, s}$, as follows. On $Y_{L,\, s+1}$ let $H$ denote the hyperplane class, and let $\mathcal E_i$, respectively $\mathcal E_L$, denote the exceptional divisor over the point $p_i$, respectively the line $L$. On $Z_{L,\, s}$ let $H_1$ and $H_2$ denote the pullbacks of the hyperplane classes from $\PP^1$ and $\PP^2$, and let $E_i$, respectively $E_L$, denote the exceptional divisor over the point $q_i$, respectively the line $L$.

\begin{lemma}
  Fixing bases for $N^1(Y)$ and $N^1(Z)$ as follows:
  \begin{align*}
    N^1(Y) = \langle H, \mathcal E_L, \mathcal E_1, \ldots, \mathcal E_{s+1} \rangle \\
    N^1(Z) = \langle H_1, H_2, E_L, E_1, \ldots E_s \rangle
  \end{align*}
  the pushforward and pullback maps $\varphi_\ast \colon N^1(Y) \arrow N^1(Z)$ and $\varphi^\ast \colon N^1(Z) \arrow N^1(Y)$ are given by the following matrices:
\begin{align}\label{formula-pushforward}
  \varphi_\ast =
  \left(
  \begin{array}{ccc|ccc}
    1 & 0 & 1 &  & & \\
    1 & 1 & 0 &  & {\mathbf 0} &\\
    -1 & -1 & -1 & &&\\
    \hline
      &  &  &  &  &\\
    & {\mathbf 0} & & & I_{s \times s} & \\
    & & & & &
  \end{array}\right)  
&\quad \quad \varphi^\ast =  \left(
  \begin{array}{ccc|ccc}
    1 & 1 & 1 &  & & \\
    -1 & 0 & -1 &  & {\mathbf 0} & \\
    0 & -1 & -1 & & & \\
    \hline
      &  &  &  &  &\\
    & {\mathbf 0} & & & I_{s \times s} & \\
    & & & & &
  \end{array}\right)  
\end{align}
\end{lemma}
\begin{proof}
Again it suffices to prove this for $s=0$. The matrix $\varphi_*$ above is obtained by identifying classes on $Y$ and $Z$ which correspond to the same ray generators, as follows. The classes $H$ on $Y$ and $H_1+H_2-E_L$ on $Z$ both correspond to the sum $e_1+\left(-(e_2+e_3) \right)$ of ray generators. The classes $\mathcal E_L$ and $H_2-E_L$ both correspond to $-(e_2+e_3)$, which corresponds to $H_2-E_L$ on $Z$. Finally, $\mathcal E_1$ and $H_1-E_L$ both correspond to $-e_1$. The matrix $\varphi^*$ is the inverse of $\varphi_*$.  
\end{proof}

Now we will compute the effective cone of $Y_{L, \, 5}$. We start with some base locus inequalities. To state these we introduce the following notation. For $i \in \{1,\ldots,5\}$, let $\Pi(L,i)$ denote the plane in $\PP^3$ spanned by $L$ and $p_i$ For distinct $i, \, j, \, k \in \{1,\ldots,5\}$, let $\Pi(i,j,k)$ denote the plane in $\PP^3$ spanned by $p_i, \, p_j, \, p_k$. 
\begin{lemma}\label{lemma-baselocusineqsyl5}
  Let $D$ be a divisor on $Y_{L,\, 5}$ with class
  \begin{align*}
    dH - m_L \mathcal E_L - \sum_{i=1}^5 m_i \mathcal E_i.
  \end{align*}
  Then the following divisors are contained in $\Bs(D)$ with the given multiplicities:
  \begin{enumerate}
  \item[(a)] $\mathcal E_L$ with multiplicity at least $\max\{0,-m_L\}$; 
  \item[(b)] $\Pi (i,j,k)$ with multiplicity at least $\max\{0,m_i+m_j+m_k-2d\}$;
  \item[(c)] $\Pi(L,i)$ with multiplicity at least $\max\{0,m_L+m_i-d\}$.
  \end{enumerate}
\end{lemma}
\begin{proof}
  These all follow from Lemma \ref{lemma-baselocusgeneral} using the following curve classes:
  
  (a): the class $e_L$ of a line in $\mathcal E_L$ that is contracted by the blowdown. Note that $\mathcal E_L~\cdot~e_L~=~-1$.
  
  (b): the class $2l-e_i-e_j-e_k$ of the proper transform of a conic in $\PP^3$ passing through $p_i$, $p_j$, and $p_k$. Note that $\Pi(i,j,k) \cdot (2l-e_i-e_j-e_k)~=~-1$.
  
  (c): the class $l-e_L-e_i$ of the proper transform of a line in $\PP^3$ intersecting $L$ and passing through $p_i$. We have $\Pi(L,i) \cdot (l-e_L-e_i)~=-1$. 
\end{proof}

To obtain more bounds on the effective cone, once more we consider the restriction of divisors to a suitable movable subvariety (\textbf{Step 4} of the method outlined in Section \ref{effconemethod}). At this point it is convenient to switch from considering effective divisors on $Y_{L, \, 5}$ to effective divisors on $Z_{L, \, 4}$; by Proposition \ref{proposition-smallmodification} these can be identified. In fact we will first consider divisors on $Z_{L, \, 5}$, but the conditions we obtain will give us useful information on $Z_{L, \, 4}$. 

  There is a $1$-parameter family of surfaces of bidegree $(2,1)$ in $\PP^1 \times \PP^2$ containing the line $L$ and the points $q_1,\ldots,q_5$. A general such surface $S$ is smooth, and its proper transform on $Z_{L,S}$ is a smooth surface $\widetilde{S}$ of class
  \begin{align*}
   2H_1+H_2-E_{L} - \sum_{i=1}^5 E_i.
  \end{align*}
By Lemma \ref{lemma-bideg21} the surface $S$ is isomorphic to $\PP^1 \times \PP^1$, so $\widetilde{S}$ is isomorphic to $X_{1,1,5}$. The effective cone of $X_{1,1,5}$ was described in Proposition \ref{proposition-x115}, and by restricting we can use this information to obtain effectivity conditions for divisors on $Z_{L,\, 5}$. To do this, we first need to record the following facts about restrictions of divisors on $Z_{L, \, 5}$ to $\widetilde{S}$. Since $\widetilde{S}$ is isomorphic to $X_{1,1,5}$, we have
  \begin{align*}
    N^1(\widetilde{S}) = \langle h_1, \, h_2, e_1, \ldots, e_5 \rangle,
  \end{align*}
  where $h_1$ and $h_2$ are the classes of the two rulings of $S \cong \PP^1 \times \PP^1$ and the $e_i$ are the exceptional curves of the blowups. Let us fix notation by choosing $h_1$ to be the ruling of $S$ containing the line $L$. Then we have
  \begin{lemma}\label{lemma-restrictions}
    The restriction map $N^1(Z_{L,\, 5}) \arrow N^1(\widetilde{S})$ is given by:
    \begin{align*}
      H_1 &\mapsto h_1 \\
      H_2 & \mapsto h_1 + h_2\\
      E_L &\mapsto h_1\\
      E_i &\mapsto e_i \quad (i=1,\ldots,5).
    \end{align*}
  \end{lemma}
  \begin{proof}
    The first two restrictions can be computed on $S$. Since $S$ has class $2H_1+H_2$ this restriction of $H_1$ to $S$ is a curve of class $(2H_1+H_2) \cdot H_1 = H_1 \cdot H_2$, which is a line in the ruling $h_1$. The restriction of $H_2$ to $S$ is a curve of class $(2H_1+H_2) \cdot H_2$; this is an effective class on $S$ with selfintersection $(2H_1+H_2) \cdot H_2^2=2$, so it must equal $h_1+h_2$.

    The exceptional divisor $E_L$ meets $\widetilde{S}$ in a smooth curve $C$; near $C$ the blow-down map $\widetilde{S} \arrow S$ is an isomorphism which maps $C$ to the line $L \subset S$, and therefore $C$ is a curve in the class $h_1$.

    The restrictions of the remaining exceptional divisors are clear.
  \end{proof}

We can now put together the necessary ingredients to determine the effective cone of the variety $Y_{L,5}$. 
\begin{theorem}\label{theorem-yl5}
  The effective cone $\Eff(Y_{L,5})$ is given by the list of generators below left and inequalities below right:
  
\begin{minipage}[t]{0.4\textwidth}
  \begin{align*}
    \rowcolors{2}{white}{gray!15}
\begin{array}{ccccccc}
H & \mathcal{E}_L & \mathcal E_1 & \mathcal E_2 & \mathcal E_3 & \mathcal E_4 & \mathcal E_5\\
\hline\hline
%\hhline{======}y
0 & 1 & 0 & 0 & 0 & 0 & 0\\
0 & 0 & 1 & 0 & 0 & 0 & 0\\
1 & -1 & -1 & 0 & 0 & 0 & 0\\
1 & 0 & -1 & -1 & -1 & 0 & 0\\
2 & -1 & -1 & -1 & -1 & -1 & -1\\
5 & -1 & -3 & -3 & -3 & -3 & -3\\
\end{array}
  \end{align*}
  \end{minipage}\quad
  \begin{minipage}[t]{0.4\textwidth}
    \begin{align*}
      \rowcolors{2}{white}{gray!15}
\begin{array}{ccccccc}
d & m_L & m_1 & m_2 & m_3 & m_4 & m_5\\
\hline\hline
%\hhline{======}
1 & 0 & 0 & 0 & 0 & 0 & 0\\
1 & 1 & 0 & 0 & 0 & 0 & 0\\
1 & 0 & 1 & 0 & 0 & 0 & 0\\
2 & 1 & 1 & 1 & 0 & 0 & 0\\
3 & 0 & 1 & 1 & 1 & 1 & 0\\
3 & 0 & 1 & 1 & 1 & 1 & 1\\
3 & 2 & 1 & 1 & 1 & 0 & 0\\
3 & 2 & 1 & 1 & 1 & 1 & 0\\
4 & 2 & 2 & 1 & 1 & 1 & 1\\
4 & 3 & 1 & 1 & 1 & 1 & 1\\
5 & 3 & 2 & 2 & 1 & 1 & 1\\
6 & 3 & 2 & 2 & 2 & 2 & 1\\
\end{array}
    \end{align*}
  \end{minipage}
\end{theorem}
\begin{proof}
  We apply the method of Section \ref{effconemethod} with the following inputs:
  \begin{itemize}
\item the base locus inequalities from Lemma \ref{lemma-baselocusexceptionals} and Lemma \ref{lemma-baselocusineqsyl5}, corresponding to the fixed divisors $\mathcal E_i$, $\mathcal E_L$, $\Pi(i,j,k)$, and $\Pi(L,i)$;
  \item the pullback via the map $Y_{L,\, 5} \arrow X_{0,3,5}$ of the inequalities from Corollary \ref{corollary-x035} cutting out the cone $\Eff(X_{0,3,5})$;
  \item the pullback via the restriction map $N^1(Z_{L,5}) \arrow N^1(\widetilde{S}$) from Lemma \ref{lemma-restrictions} of inequalities from Lemma \ref{proposition-x115} cutting out the cone $\Eff(X_{1,1,5})$.
  \end{itemize}
Note that the last bullet point gives inequalities bounding $\Eff(Z_{L, \, 5})$. We convert these into inequalities bounding $\Eff(Y_{L,6})$ by using the dual of the map $\varphi_*$ given in (\ref{formula-pushforward}). These inequalities turn out not to involve one of the multiplicities $m_i$, and so we can view them as inequalities on  $\Eff(Y_{L,5})$.
  
The Normaliz files \texttt{YL5-ineqs} compute the cone defined by these inequalities. We then add the classes of the fixed divisors listed in the first bullet point above, using the files \texttt{YL5-gens}, to obtain the list of extremal rays and inequalities  shown in the tables above. It remains to show that all the extremal rays in the left table are effective.

The first 4 rows in the table correspond to effective divisors of type $\mathcal E_L$, $\mathcal E_i$, $\Pi(L,i)$, and $\Pi(i,j,k)$.

For Row 5, counting dimensions shows that there is a 1-parameter family of quadric surfaces in $\PP^3$ containing the line $L$ and the 5 points: containment of $L$ imposes $3$ conditions and passage through $5$ points in general position imposes $5$ conditions. The proper transform of any such quadric has class
\begin{align*}
2H-\mathcal E_L - \sum_{i=1}^5 \mathcal E_i
\end{align*}
which shows that Row 5 is effective.

It remains to deal with the final row of our table, that is to prove
that the class
      \begin{align*}
       D &= 5H-\mathcal E_L - 3 \sum_{i=1}^5 \mathcal E_i
      \end{align*}
      is effective. The virtual dimension of $|D|$ is
      \begin{align*}
        {8 \choose 3} - 6 - 5 \cdot 10 -1&=-1
      \end{align*}
so this does not follow from a simple parameter count. Instead we will prove the this class is effective by restriction to a suitable subvariety.

As already seen, there is a 1-parameter family of quadric surfaces in $\PP^3$ containing the line $L$ and the 5 points. By generality of $L$ and the points, every such quadric is smooth; let $Q$ be the proper transform on $Y_{L, \, 5}$ of any such quadric. Restricting $D$ to $Q$ gives the short exact sequence of sheaves
      \begin{align*}
        0 \arrow O(D-Q) \arrow O(D) \arrow O(D|_{Q}) \arrow 0
      \end{align*}
      and the corresponding long exact sequence of cohomology
       \begin{align*}
        0 \arrow H^0(Y,O(D-Q)) \arrow H^0(Y,O(D)) \arrow H^0(Q,O(D|_{Q})) \arrow H^1(Y,O(D-Q)) \arrow \cdots
      \end{align*}

       Now $Q$ has class $2H-\mathcal E_L - \sum_i \mathcal E_i$, so we get $D-Q=3H-2 \sum_i \mathcal E_i$ which is not effective since a cubic surface can be singular at no more than 4  points in general position. So $h^0(Y,O(D-Q))=0$.

       On the other hand, according to \cite[Remark 11.1]{SAGA2014} we have $H^i(Y,O(D-Q))=0$ for $i \geq 2$, so
       \begin{align*}
         \chi(Y,O(D-Q)) &= h^0(Y,O(D-Q)) - h^1(Y,O(D-Q)).
       \end{align*}
       Using Lemma \ref{vdimeuler} we have
       \begin{align*}
         \chi(Y,O(D-Q))&=\vdim|D-Q|+1\\
         &=20 - 5 \cdot 4 =0,
       \end{align*}
       so $h^0(Y,O(D-Q))=h^1(Y,O(D-Q))=0$.
       
       The exact sequence of cohomology above therefore gives an isomorphism 
       \begin{align*}
         H^0(Y,O(D)) &\cong H^0(Q,O(D|_{Q})).
       \end{align*}
      Proposition \ref{proposition-x115} tells us that  $D|_{Q}=4h_1+5h_2-3\sum_i E_i$ is effective, so we conclude that $D$ is effective also. 
\end{proof}
\begin{remark}
  Since the class $5H-\mathcal E_L - 3 \sum_i E_i$ is a primitive generator of an extremal ray of $\Eff(Y_{L,5})$, it must be represented by an irreducible quintic surface $S \subset \PP^3$. If $R$ is a twisted cubic incident to $L$ and passing through $p_1,\ldots,p_5$, we compute that $S \cdot R =-1$. so $R$ is contained in $D$. As we move the point $R \cap L$ these twisted cubics sweep out an irreducible surface, which must therefore equal $S$. 
\end{remark}
Using the linear map $\varphi_*$ given in (\ref{formula-pushforward}) and its dual, we can restate this result in terms of the variety $Z_{L,\, 4}$; we will use the information in this form in the next section.
\begin{corollary}\label{corollary-zl4}
The effective cone $\Eff(Z_{L,4})$ is given by the generators below left and the inequalities below right.  
  
\begin{minipage}[t]{0.4\textwidth}
  \begin{align*}
    \rowcolors{2}{white}{gray!15}
\begin{array}{ccccccc}
H_1 & H_2 & E_L & E_1 & E_2 & E_3 & E_4\\
\hline\hline
%\hhline{======}
0 & 0 & 1 & 0 & 0 & 0 & 0\\
0 & 0 & 0 & 1 & 0 & 0 & 0\\
1 & 0 & 0 & -1 & 0 & 0 & 0\\
0 & 1 & -1 & 0 & 0 & 0 & 0\\
0 & 1 & 0 & -1 & -1 & 0 & 0\\
1 & 1 & -1 & -1 & -1 & -1 & 0\\
2 & 4 & -1 & -3 & -3 & -3 & -3
\end{array}
  \end{align*}
  \end{minipage}\quad
  \begin{minipage}[t]{0.4\textwidth}
    \begin{align*}
      \rowcolors{2}{white}{gray!15}
\begin{array}{ccccccc}
d_1 & d_2 & m_L & m_1 & m_2 & m_3 & m_4\\
\hline\hline
%\hhline{======}
1 & 0 & 0 & 0 & 0 & 0 & 0\\
0 & 1 & 0 & 0 & 0 & 0 & 0\\
1 & 1 & 1 & 0 & 0 & 0 & 0\\
1 & 1 & 0 & 1 & 0 & 0 & 0\\
1 & 1 & 1 & 1 & 0 & 0 & 0\\
1 & 2 & 0 & 1 & 1 & 0 & 0\\
1 & 2 & 0 & 1 & 1 & 1 & 0\\
1 & 2 & 1 & 1 & 1 & 0 & 0\\
1 & 3 & 0 & 1 & 1 & 1 & 1\\
1 & 3 & 1 & 1 & 1 & 1 & 0\\
1 & 3 & 1 & 1 & 1 & 1 & 1\\
2 & 2 & 0 & 1 & 1 & 1 & 1\\
2 & 3 & 0 & 2 & 1 & 1 & 1\\
2 & 3 & 1 & 2 & 1 & 1 & 1\\
2 & 4 & 1 & 2 & 2 & 1 & 1\\
3 & 2 & 2 & 1 & 1 & 1 & 0\\
3 & 2 & 2 & 1 & 1 & 1 & 1\\
3 & 3 & 3 & 1 & 1 & 1 & 1\\
3 & 4 & 2 & 2 & 2 & 2 & 1\\
3 & 5 & 2 & 2 & 2 & 2 & 2
\end{array}
    \end{align*}
  \end{minipage}
\end{corollary}

\subsection*{The effective cone of $Y_{L,\, 6}$}
Finally for this section, we will consider $Y_{L,\, 6}$, the blowup of $\PP^3$ in a line and 6 points. The effective cone of this variety will be an important input into the computations of the next section. To compute this effective cone, it will be necessary to use the small modification described in Proposition \ref{proposition-smallmodification} and obtain restrictions on divisors on both $Y_{L,\, 6}$ and $Z_{L,\, 5}$; putting together all the resulting restrictions, we can then find the effective cone.

We start with two more base locus lemmas.

\begin{lemma}\label{lemma-baselocusineqsyl6}
  Given a class in $N^1(Y_{L,6})$ of the form
$
2H- \sum_{i=1}^6 E_i - E_j,
$
there is a unique effective divisor $\mathcal F_j$ with the given class.

For a divisor $D$ on $Y_{L,6}$ with class
\begin{align*}
dH-m_L E_L - \sum_{i=1}^6 m_i E_i,
\end{align*}
 the divisor $\mathcal F_j$ is contained in $\Bs(D)$ with multiplicity at least 
$\max\left\{0,\sum_{i=1}^6 m_i + m_j -4d\right\}.$
%\eli{added max}
\end{lemma}
\begin{proof}
  To prove the first claim, we observe that the class $2H-\sum_{i=1}^6 E_i -E_j$ is represented by the proper transform of a quadric surface in $\PP^3$ with a singular point at $p_j$. By projection away from $p_j$ there is a unique such quadric $\mathcal F_j$.

  The second claim follows from \cite[Lemma 4.1]{BDP16}; one can also see it directly as follows. Let $Q$ be the proper transform of a smooth quadric in $\PP^3$ passing through the points $p_1,\ldots,p_6$. Then the intersection $\mathcal F_j \cap Q$ is a curve $C_j$ of class $4l-\sum_i e_i - e_j$. There is a 2-parameter family of such curves $C_j$, and the union of all such curves equals $\mathcal F_j$. 
  The only possibility for a reducible such curve is $C_j = R \cup L_j$, where $R$ is the proper transform of the twisted cubic through $p_1,\ldots,p_6$ and $L_j$ is the proper transform of a line on $F_j$ through $p_j$. So there is a 1-parameter family of such reducible curves, and therefore the general such $C_j$ is irreducible.

  For a divisor $D$ on $Y_{L,6}$ with class $dH-m_L \mathcal E_L -\sum_{i=1}^5 m_i \mathcal E_i$, the intersection number of this divisor with $C_j$ equals
  \begin{align*}
    D \cdot C_j &= 4d - \sum_{i=1}^6 m_i - m_j.
  \end{align*}
  Finally we compute that $\mathcal F_j \cdot C_j =-1$, so by Lemma \ref{lemma-baselocusgeneral}, the divisor $F_j$ is contained in the base locus $\Bs(D)$ with multiplicity at least
  \begin{align*}
    -D \cdot C_j &= \sum_{i=1}^6m_i+m_j-4d.  
  \end{align*}
\end{proof}

\begin{lemma}\label{lemma-baselocuszl5}
  Given a class in $N^1(Z_{L,\, 5})$ of the form
$
    H_1+H_2-E_L-E_i-E_j-E_k
$
  there is a unique effective divisor $G_{ijk}$ with the given class. 
  
For a divisor $D$ on $Z_{L,\, 5}$ with class
$
d_1H_1+d_2H_2-m_L E_L - \sum_{i=1}^5 m_i E_i
$
 the divisor $G_{ijk}$ is contained in $\Bs(D)$ with multiplicity at least $$\max\left\{0,m_i+m_j+m_k -d_1-2d_2\right\}.$$
\end{lemma}
\begin{proof}
  For the first claim, the given class is represented by the proper transform of a divisor of bidegree $(1,1)$ containing the line $L$ and the points $p_i$, $p_j$, $p_k$. Containing the line $L$ imposes 2 conditions on such divisors, and passing through a point imposes 1 condition. So the corresponding linear system on $Z_{L,\ 5}$ has dimension $0$, meaning that it contains there is a unique effective divisor $G_{ijk}$.

  For the second claim, consider a general divisor $H$ with class $H_1+H_2-E_i-E_j-E_k$. The intersection $G_{ijk} \cap H$ will then be a curve of class $C_{ijk} = l_1+2l_2-e_L-e_i-e_j-e_k$.

  For a divisor $D$ on $Z_{L, \, 5}$ with class $d_1H_1+d_2H_2-m_LE_L-\sum_{i=1}^5 m_i E_i$, the intersection number of $D$ with the curve $C_{ijk}$ equals
  \begin{align*}
    D \cdot C_{ijk} &= d_1+2d_2-m_L - m_i-m_j-m_k.
  \end{align*}
  Finally we compute that $G_{ijk} \cdot C_{ijk} = -1$, so by Lemma \ref{lemma-baselocusgeneral} the divisor $G_{ijk}$ is contained in the base locus $Bs(D)$ with multiplicity at least
  \begin{align*}
    -D \cdot C_{ijk} &= m_i+m_j+m_k -d_1-2d_2.
  \end{align*}
\end{proof}

Now we return to consider divisors on $Y_{L,\, 6}$. As in the computation of $\Eff(Y_{L,5})$, we will need to use a restriction map to obtain extra inequalities on divisors on $\Eff(Y_{L,6})$. For a general line $L$ and $6$ points in general position in $\PP^3$, there is a unique smooth quadric $Q$ containing all of them. Let $\widetilde{Q}$ be the proper transform of $Q$ on $Y_{L,6}$. If $D$ is any irreducible effective divisor on $Y_{L,6}$ other than $\widetilde{Q}$ itself, then the restriction $D|_{\widetilde{Q}}$ must be effective.

The divisor $\widetilde{Q}$ is isomorphic to $X_{1,1,6}$, so the inequalities from Proposition \ref{proposition-x116} will give us restrictions on the class of $D|_{\widetilde{Q}}$. We can label classes on $\widetilde{Q}$ so that $\mathcal E_L$ restricts to $h_1$. The restriction map $N^1(Y_{L,6}) \arrow N^1(X_{1,1,6})$ is then evidently given as follows:
\begin{lemma}\label{lemma-restrictions3}
  The restriction map $N^1(Y_{L,6}) \arrow N^1(\widetilde{Q})$ is given by
  \begin{align*}
    H &\mapsto h_1+h_2\\
    \mathcal E_L &\mapsto h_1\\
    \mathcal E_i &\mapsto e_i
  \end{align*}
\end{lemma}
We can now compute the effective cone of this threefold.
\begin{theorem}\label{theorem-yl6}
  The effective cone $\Eff(Y_{L,6})$ is given by the generators below left and the inequalities below right.

  \begin{minipage}[t]{0.4\textwidth}
  \begin{align*}
    \rowcolors{2}{white}{gray!15}
\begin{array}{cccccccc}
H & \mathcal{E}_L & \mathcal{E}_1 & \mathcal{E}_2 & \mathcal{E}_3 & \mathcal{E}_4 & \mathcal{E}_5 & \mathcal{E}_6\\
\hline\hline
%\hhline{======}
0 & 1 & 0 & 0 & 0 & 0 & 0 & 0\\
0 & 0 & 1 & 0 & 0 & 0 & 0 & 0\\
1 & -1 & -1 & 0 & 0 & 0 & 0 & 0\\
1 & 0 & -1 & -1 & -1 & 0 & 0 & 0\\
2 & 0 & -2 & -1 & -1 & -1 & -1 & -1\\
2 & -1 & -1 & -1 & -1 & -1 & -1 & -1\\
5 & -1 & -3 & -3 & -3 & -3 & -3 & 0\\
\end{array}
  \end{align*}
    \end{minipage} \quad
    \begin{minipage}[t]{0.4\textwidth}
      \begin{align*}
            \rowcolors{2}{white}{gray!15}
  \begin{array}{cccccccc}
    d & m_L & m_1 & m_2 & m_3 & m_4 & m_5 & m_6\\
\hline\hline
%\hhline{======}
1 & 0 & 0 & 0 & 0 & 0 & 0 & 0\\
1 & 1 & 0 & 0 & 0 & 0 & 0 & 0\\
1 & 0 & 1 & 0 & 0 & 0 & 0 & 0\\
2 & 1 & 1 & 1 & 0 & 0 & 0 & 0\\
3 & 0 & 1 & 1 & 1 & 1 & 0 & 0\\
3 & 0 & 1 & 1 & 1 & 1 & 1 & 0\\
3 & 2 & 1 & 1 & 1 & 0 & 0 & 0\\
3 & 2 & 1 & 1 & 1 & 1 & 0 & 0\\
4 & 2 & 2 & 1 & 1 & 1 & 1 & 0\\
4 & 3 & 1 & 1 & 1 & 1 & 1 & 0\\
5 & 0 & 2 & 2 & 1 & 1 & 1 & 1\\
5 & 2 & 2 & 2 & 1 & 1 & 1 & 1\\
5 & 3 & 2 & 2 & 1 & 1 & 1 & 0\\
5 & 4 & 1 & 1 & 1 & 1 & 1 & 1\\
6 & 3 & 2 & 2 & 2 & 2 & 1 & 0\\
6 & 3 & 3 & 2 & 1 & 1 & 1 & 1\\
7 & 0 & 2 & 2 & 2 & 2 & 2 & 2\\
7 & 2 & 2 & 2 & 2 & 2 & 2 & 2\\
7 & 4 & 3 & 3 & 1 & 1 & 1 & 1\\
9 & 3 & 3 & 3 & 3 & 3 & 2 & 1\\
11 & 4 & 4 & 4 & 3 & 3 & 3 & 1\\
16 & 5 & 5 & 5 & 5 & 5 & 5 & 2
\end{array}
  \end{align*}
\end{minipage}
  \end{theorem}
\begin{proof}
  We apply the method of Section \ref{effconemethod} with the following inputs:
  \begin{itemize}
  \item the base locus inequalities from Lemma \ref{lemma-baselocusexceptionals}, Lemma \ref{lemma-baselocusineqsyl5}, Lemma \ref{lemma-baselocusineqsyl6}, corresponding to the fixed divisors $\mathcal E_i$, $\mathcal E_L$, $\Pi(i,j,k)$, $\Pi(L,i)$, and $\mathcal F_i$ on $Y_{L,\, 6}$;
    \item the base locus inequality from Lemma \ref{lemma-baselocuszl5}, corresponding to the fixed divisors $G_{ijk}$ on $Z_{L,\, 5}$;
  \item the pullback via the restriction map $N^1(Y_{L,6}) \arrow N^1(\widetilde{Q})$ from Lemma \ref{lemma-restrictions3} of the inequalities from Proposition \ref{proposition-x116};
  \item the pullback via the morphism $Y_{L,6} \arrow Y_{L,5}$ of the inequalities from Theorem \ref{theorem-yl5};
 \item the pullback via the morphism $Y_{L,6} \arrow X_{0,3,6}$ of the inequalities from Theorem \ref{theorem-x036};
    \item the pullback via the morphism $Z_{L,5} \arrow X_{1,2,5}$ of the inequalities from Theorem \ref{Eff 125}.
  \end{itemize}

Again, in cases where we obtain an inequality on divisors on $Z_{L,\, 5}$, we use the dual of the isomorphism $\varphi \colon N^1(Y) \arrow N^1(Z)$ given in (\ref{formula-pushforward}) to convert these into inequalities on divisors on $Y_{L,\, 6}$. 

The Normaliz files {\tt YL6-ineqs} compute the cone defined by these inequalities. We then add the classes of the fixed divisors listed in the first bullet point above, using the files \texttt{YL6-gens} to obtain the list of extremal rays and inequalities  shown in the tables above. It remains to show that all the extremal rays in the left table are effective.

All the classes in the table on the left-hand side are pullbacks of effective classes on $Y_{L,5}$, with the exception of the class $2H-\mathcal E_L - \sum_{i=1}^6 \mathcal E_i$.
This class is represented by the proper transform of the unique quadric in $\PP^3$ containing the line $L$ and the points $p_i$, hence it is effective.
\end{proof}

\section{Blowups of $\PP^1 \times \PP^3$}\label{section-dim4}
In this section we use the computations of the previous section as inputs into our method, in order to compute the effective cone of divisors of the blowup of $\PP^1 \times \PP^3$ in up to 6  points in general position.

\subsection*{4 points}
First we compute the effective cone of the variety $X_{1,3,4}$, which by definition is the blowup of $\PP^1 \times \PP^3$ in a  set of 4 points in general position. Again this result follows from Hausen--S\"u{\ss} \cite[Theorem 1.2]{HS10}. It could also be proved using the method of Section \ref{effconemethod}, starting from the toric case $X_{1,3,2}$ and iterating, but for brevity we give a more direct proof. 

\begin{theorem}\label{theorem-x134}
  The effective cone $\Eff(X_{1,3,4})$ is given by the list of generators below left and inequalities below right:
  
\begin{minipage}[t]{0.4\textwidth}
  \begin{align*}
    \rowcolors{2}{white}{gray!15}
\begin{array}{cccccc}
H_1 & H_2 & E_1 & E_2 & E_3 & E_4\\
\hline\hline
%\hhline{======}
0 & 0 & 1 & 0 & 0 & 0\\
1 & 0 & -1 & 0 & 0 & 0\\
0 & 1 & -1 & -1 & -1 & 0\\
\end{array}  
  \end{align*}
  \end{minipage}\quad
  \begin{minipage}[t]{0.4\textwidth}
    \begin{align*}
      \rowcolors{2}{white}{gray!15}
\begin{array}{ccccccc}
d_1 & d_2 & m_1 & m_2 & m_3 & m_4 \\
\hline\hline
%\hhline{======}
1 & 0 & 0 & 0 & 0 & 0 \\
0 & 1 & 0 & 0 & 0 & 0 \\
1 & 1 & 1 & 0 & 0 & 0 \\
1 & 2 & 1 & 1 & 0 & 0 \\
1 & 3 & 1 & 1 & 1 & 0 \\
1 & 3 & 1 & 1 & 1 & 1 \\
\end{array}
  \end{align*}
  \end{minipage}  
\end{theorem}
\begin{proof}
  The classes listed in the table on the left are clearly effective, so they span a subcone $K$ of the cone $\Eff(X_{1,3,4})$. The Normaliz files {\tt X134-gens} shows that the subcone $K$ is cut out by the inequalities listed in the table on the right. We will prove that any effective divisor on $X_{1,3,4}$ must satisfy all the inequalities listed in the table on the right, which gives the reverse inclusion $\Eff(X_{1,3,4}) \subset K$.

  To prove this, it is enough to show that for each inequality in the table on the right, there is a curve class $C_i \in N_1(X)$ such that the given inequality is of the form $D \cdot C_i \geq 0$ and such that irreducible curves in $C_i$ cover a Zariski-dense open subset of $X_{1,3,4}$.

  For the first two inequalities in the table, the curve classes we need are $C_1=l_1$ and $C_2=l_2$, which clearly have the required properties.

  For the remaining rows we can use the following argument. We give the details for the last row; the other rows are similar but easier. For any embedding $f \colon \PP^1 \arrow \PP^3$ whose image is a twisted cubic, its graph $\Gamma_f$ is an irreducible smooth curve of bidegree $(1,3)$ in $\PP^1 \times \PP^3$. Such a morphism $f$ is given by a choice of 4 linearly independent cubic forms in 2 variables, and direct computation shows that for 5  points $p_1, \ldots, p_5$ in general position in $\PP^1 \times \PP^3$, cubic forms can be chosen appropriately to make $\Gamma_f$ pass through all the $p_i$. Therefore, blowing up at $p_1,\ldots,p_4$ and taking proper transforms, we get a family of irreducible curves with class $l_1+3l_2-e_1-\cdots-e_4$ which cover a Zariski-dense open set, as required.
\end{proof}

\subsection*{5 points}
Next we consider the variety $X_{1,3,5}$, which by definition is the blowup of $\PP^1 \times \PP^3$ in a set of 5 points in general position. To compute the effective cone, we will again use base locus inequalities coming from fixed divisors, together with inequalities pulled back from blowups in fewer points. Additionally, we will iterate the strategy that we used in the previous section: we restrict divisors on $X_{1,3,5}$ to a subvariety whose effective cone we computed in Section \ref{section-dim3}, in order to obtain extra effectivity conditions, following \textbf{Step 4} of the cone method of Section \ref{effconemethod}.

To do this, let $s$ be a set of points in general position in $\PP^1 \times \PP^3$, let $V$ be a smooth divisor of bidegree $(1,1)$ containing all the points, and let $\widetilde{V_s}$ be the proper transform of $V$ on $X_{1,3,s}$. Then by Lemma \ref{lemma-divisor1} we have that $\widetilde{V_s}$ is isomorphic to $Y_{L,s}$. 

\begin{lemma}\label{lemma-restriction2}
 The restriction map $N^1(X_{1,3,s}) \arrow N^1(\widetilde{V_s})$ is given by  
 \begin{align*} 
   H_1 & \mapsto H-\mathcal E_L\\
   H_2 & \mapsto H\\
   E_i &\mapsto \mathcal E_i
 \end{align*}
\end{lemma}
\begin{proof}
It suffices to prove the case $s=0$.   
\end{proof}

We also need one final base locus lemma. For this, let $i, \, j, \, k$ be any set of 3 distinct indices in $\{1,\ldots,6\}$. Since there is a unique plane in $\PP^3$ containing 3  points in general position, there is a unique effective divisor on $X_{1,3,6}$ in the class $H_2-E_i-E_j-E_k$. We denote this divisor by $\Pi(i,j,k)$. 
\begin{lemma}\label{lemma-baselocus3pts}
Fix $s \geq 3.$ For a divisor $D$ on $X_{1,3,s}$ with class $d_1H_1+d_2H_2-\sum_{i=1}^s m_i E_i$. The divisor $\Pi(i,j,k)$ is contained in $\Bs(D)$ with multiplicity at least
  \begin{align*}
    \max\{0,m_i+m_j+m_k - d_1 - 2d_2\}. 
  \end{align*}
\end{lemma}
\begin{proof}
Choose 2 general hypersurfaces of bidegree $(1,1)$ passing through the 3 points $p_i, \, p_j, \, p_k$, and let $D_1, \, D_2$ be their proper transforms on $X_{1,3,s}$. Let $C$ be the curve $\Pi(i,j,k) \cap D_1 \cap D_2$. By Proposition \ref{intersection-table} we compute
\begin{align*}
  C \cdot H_1 &= (H_2-E_i-E_j-E_k)\cdot(H_1+H_2-E_i-E_j-E_k)^2 \cdot H_1= 1,\\
  C \cdot H_2 &=  (H_2-E_i-E_j-E_k)\cdot(H_1+H_2-E_i-E_j-E_k)^2 \cdot H_2= 2,
\end{align*}
 so $C$ is a curve of class $l_1+2l_2-e_i-e_j-e_k$ and therefore we have $C \cdot \Pi(i,j,k)=-1$.

 By choosing the divisors $D_1$ and $D_2$ appropriately, we can cover the divisor $\Pi(i,j,k)$ by such curves $C$, so by Lemma \ref{lemma-baselocusgeneral} the claim follows. 
\end{proof}
Now we have the ingredients we need for our computation of the effective cone of $X_{1,3,5}$.
\begin{theorem}\label{theorem-x135}
  The effective cone $\Eff(X_{1,3,5})$ is given by the list of generators below left and inequalities below right:
  
\begin{minipage}[t]{0.4\textwidth}
  \begin{align*}
    \rowcolors{2}{white}{gray!15}
\begin{array}{ccccccc}
H_1 & H_2 & E_1 & E_2 & E_3 & E_4 & E_5\\
\hline\hline
%\hhline{======}
0 & 0 & 1 & 0 & 0 & 0 & 0\\
1 & 0 & -1 & 0 & 0 & 0 & 0\\
0 & 1 & -1 & -1 & -1 & 0 & 0\\
1 & 1 & -1 & -1 & -1 & -1 & -1\\
1 & 4 & -3 & -3 & -3 & -3 & -3
\end{array}  
  \end{align*}
  \end{minipage}\quad
  \begin{minipage}[t]{0.4\textwidth}
    \begin{align*}
      \rowcolors{2}{white}{gray!15}
\begin{array}{ccccccc}
d_1 & d_2 & m_1 & m_2 & m_3 & m_4 & m_5\\
\hline\hline
%\hhline{======}
1 & 0 & 0 & 0 & 0 & 0 & 0\\
0 & 1 & 0 & 0 & 0 & 0 & 0\\
1 & 1 & 1 & 0 & 0 & 0 & 0\\
1 & 2 & 1 & 1 & 0 & 0 & 0\\
1 & 3 & 1 & 1 & 1 & 0 & 0\\
1 & 3 & 1 & 1 & 1 & 1 & 0\\
1 & 4 & 1 & 1 & 1 & 1 & 1\\
2 & 4 & 2 & 1 & 1 & 1 & 1\\
2 & 5 & 2 & 2 & 1 & 1 & 1\\
3 & 3 & 1 & 1 & 1 & 1 & 1\\
3 & 6 & 2 & 2 & 2 & 2 & 1
\end{array}
  \end{align*}
  \end{minipage}  
\end{theorem}
\begin{proof}
  We apply the method of Section \ref{effconemethod} with the following inputs:
  \begin{itemize}
  \item the base locus inequalities from Lemma \ref{lemma-baselocusexceptionals} and Lemma \ref{lemma-baselocus3pts} corresponding to the fixed divisors $E_i$, $H_1-E_i$, and $H_2-E_i-E_j-E_k$;
  \item the pullback via the morphism $X_{1,3,5} \arrow X_{1,3,4}$ of the inequalities from Theorem \ref{theorem-x134} cutting out the cone $\Eff(X_{1,3,4})$;
  \item the pullback via the restriction map $N^1(X_{1,3,5}) \arrow N^1(\widetilde{V_5})$ from Lemma \ref{lemma-restriction2} of the inequalities from Theorem \ref{theorem-yl5} cutting out the cone $\Eff(\widetilde{V_5})$.
    \end{itemize}
  The Normaliz files \texttt{X135-ineqs}
compute the restricted cone defined by these inequalities. Adding the fixed divisors $E_i$, $H_1-E_i$, and $H_2-E_i-E_j-E_k$, the files \texttt{X135-gens} then compute the list of extremal rays and inequalities shown in the tables above. Lemma \ref{vdim lower bound} shows that all generators in the table on the left except the last are represented by effective divisors. It remains to the prove that the last generator $H_1+4H_2-3 \sum_{i=1}^5 E_i$ is represented by an effective divisor.

  We will show by direct computation that for 5  points $p_1,\ldots,p_5$ in general position in $\PP^1 \times \PP^3$, there is a hypersurface $D$ of bidegree $(1,4)$ with multiplicity 3 at each of the $p_i$. The proper transform of $D$ on $X_{1,3,5}$ will then be an effective divisor with the required class. Denote the homogeneous coordinates on $\PP^1$ by $s, \, t$ and those on $\PP^3$ by $w, \, x, \, y, \, z$. Using projective transformations we may assume that the points $p_i$ have homogeneous coordinates
  \begin{align*}
    p_1 = \left([1,0], \, [1,0,0,0] \right), \quad      p_2 = \left([0,1], \, [0,1,0,0] \right)&, \quad    p_3 = \left([1,1], \, [0,0,1,0] \right),\\  p_4 = \left([1,a], \, [0,0,0,1] \right), \quad   p_5 = &\left([1,b], \, [1,1,1,1] \right),
  \end{align*}
  where $a$ and $b$ are distinct complex numbers different from $0$ and $1$.

  Define
    \begin{minipage}[t]{0.4\textwidth}
  \begin{align*}
    C_1&=(w-x)(w-y)z\\
    C_2&=(w-x)(w-z)y\\
    C_3&=(w-y)(w-z)x\\
  \end{align*}
    \end{minipage}
    \begin{minipage}[t]{0.4\textwidth}
  \begin{align*}
    L_1&=(b-1)sx+(s-t)y\\
    L_2&=(a-b)sx+(t-as)z\\
  \end{align*}
    \end{minipage}
    
  Then each $C_i$ is a form of bidegree $(0,3)$ with multiplicity $1$ at $p_1$ and multiplicity 2 at the other $p_i$, while $L_1$ and $L_2$ are forms of bidegree $(1,1)$ with multiplicity 1 at each $p_i$. So for any coefficients $a_{ij} \in \CC$ the form
  \begin{align*}
  F&=  \sum_{i=1}^2 \sum_{j=1}^3 a_{ij} L_i C_j
  \end{align*}
 has bidegree $(1,4)$ and multiplicity at least $2$ at $p_1$ and $3$ at the other $p_i$. We can then attempt to choose the coefficients $a_{ij}$ such that the form $F$ has multiplicity 3 at $p_1$ also. Computing, one finds that this occurs if and only if $F$ is a multiple of the form
  \begin{align*}
   F_0 = aL_1C_1 +(b-a)L_1C_3 + L_2C_2+(b-1)L_2C_3.
  \end{align*}
So the hypersurface $D=\left\{F_0=0\right\}$ has bidegree $(1,4)$ and multiplicity $3$ at each point $p_i$ as required.
  \end{proof}

We conclude this section with the following discussion on the ray generator of bidegree $(1,4)$.
\begin{proposition}\label{divisor 1433333 is fixed}
The divisor $H_1+4H_2-3 \sum_{i=1}^5 E_i$ on $X_{1,3,5}$ is fixed.
\end{proposition}
\begin{proof}
  As before let $\widetilde{V}_5$ be the proper transform on $X_{1,3,5}$ of a general hypersurface of bidegree $(1,1)$ passing through the 5 points. Recall that  $\widetilde{V}_5$ is isomorphic to $Y_{L,5}$, the blowup of $\PP^3$ in a line and 5 points.

  Twist the ideal sheaf sequence for $\widetilde{V}_5$ by $D=H_1+4H_2-3\sum_{i=1}^5 E_i$ to get the short exact sequence
  \begin{align*}\tag{$\ast$}
    0 \arrow O_X\left(3H_2-2\sum_{i=1}^5 E_i\right) \arrow O_X(D) \arrow O_{\widetilde{V}_5}\left(D|_{\widetilde{V}_5}\right) \arrow 0
  \end{align*}
  By Lemma \ref{lemma-restriction2} the restriction of $D$ to  $\widetilde{V}_5$ gives the class $3H-\mathcal E_L -\sum_{i=1}^5 \mathcal E_i$, which is fixed as we saw in the proof of Theorem \ref{theorem-yl5}. So $H^0(O_{\widetilde{V}_5}(D))=1$. On the other hand, any effective divisor on $X_{1,3,5}$ with class $3H_2-2\sum_{i=1}^5 E_i$ must come from a cubic in $\PP^3$ with singularities at 5 points in general position. No such cubic exists, so we have $H^0(O_X(3H_2-2\sum_{i=1}^5 E_i))=0$. So the long exact sequence of cohomology associated to the short exact sequence ($\ast$) shows that $H^0(O_X(D))=1$ as required.
\end{proof}

\begin{remark}\label{Laface-Moraga1}
  %%{\color{blue}
  As observed in Remark \ref{Laface-Moraga0}, a triple point imposes a number of conditions to the family of bidegree $(1,4)$-hypersurfaces of $\PP^1\times \PP^n$ that is one less than the expected one obtained by a parameter count. This is geometrically justified by the presence of the fixed line
$C_i=l_1-e_i$ in the base locus of $D$ at least $-D\cdot C_i$ times. %}
If, for instance, we consider $D=H_1+4H_2-3\sum_{i=1}^5E_i$ on $X_{1,2,5}$, we have  $D\cdot C_i=-2$, for $i=1,\dots,5$, and each  $C_i$ contributes by $1$ to the fibre-expected dimension formula: 
$$\textrm{fibre-dim}|D|=\vdim|D|+5=-1.$$

Furthermore, we know from Proposition \ref{divisor 1433333 is fixed} that $\dim|D|=0$, so there is a gap of $1$ between the dimension and the fibre-expected dimension.
Our expectation is that the fixed curve $C=l_1+3l_2-\sum_i^5 e_i$, that satisfies $D \cdot C =-2$, contributes by $1$ to the dimension count, so that 
$$\dim|D|= \textrm{fibre-dim}|D|+1.$$
\end{remark}
% \marginpar{\tiny{\color{blue}I shortened this remark and moved part of it to section on $X_{126}$}
%}

\subsection*{6 points in $\PP^1 \times \PP^3$}
Now we come to our final effective cone computation. Fix a set of 6 points in general position in  $\PP^1 \times \PP^3$ and let $X_{1,3,6}$ be the corresponding blowup. We will compute the effective cone $\Eff(X_{1,3,6})$ in a similar way to the previous case.

A parameter count shows that there is a 1-parameter family of divisors of bidegree $(1,1)$ passing through the 6 points. Let $V_6$ be a smooth such divisor, and let $\widetilde{V_6} \cong Y_{L,6}$ be the proper transform of $V_6$ on $X_{1,3,6}$. Then as before, using the restriction map $N^1(X_{1,3,6}) \arrow N^1(\widetilde{V_6})$  given in Lemma \ref{lemma-restriction2}, we can pull back the inequalities cutting out $\Eff(\widetilde{V_6})$ to get inequalities cutting out $\Eff(X_{1,3,6})$.

Together with our existing base locus lemmas and pulling back from $X_{1,3,5}$, this gives us enough information to compute our final effective cone.
\begin{theorem}\label{theorem-x136}
  The effective cone $\Eff(X_{1,3,6})$ is given by the generators below left and the inequalities below right.

  \begin{minipage}[t]{0.4\textwidth}
  \begin{align*}
\rowcolors{2}{white}{gray!15}
\begin{array}{cccccccc}
H_1 & H_2 & E_1 & E_2 & E_3 & E_4 & E_5 & E_6\\
\hline\hline
%\hhline{======}
0 & 0 & 1 & 0 & 0 & 0 & 0 & 0\\
1 & 0 & -1 & 0 & 0 & 0 & 0 & 0\\
0 & 1 & -1 & -1 & -1 & 0 & 0 & 0\\
0 & 2 & -2 & -1 & -1 & -1 & -1 & -1\\
1 & 1 & -1 & -1 & -1 & -1 & -1 & -1\\
1 & 4 & -3 & -3 & -3 & -3 & -3 & 0
\end{array}
  \end{align*}
    \end{minipage} \quad
    \begin{minipage}[t]{0.4\textwidth}
      \begin{align*}
            \rowcolors{2}{white}{gray!15}
\begin{array}{cccccccc}
d_1 & d_2 & m_1 & m_2 & m_3 & m_4 & m_5 & m_6\\
\hline\hline
%\hhline{======}
1 & 0 & 0 & 0 & 0 & 0 & 0 & 0\\
0 & 1 & 0 & 0 & 0 & 0 & 0 & 0\\
1 & 1 & 1 & 0 & 0 & 0 & 0 & 0\\
1 & 2 & 1 & 1 & 0 & 0 & 0 & 0\\
1 & 3 & 1 & 1 & 1 & 0 & 0 & 0\\
1 & 3 & 1 & 1 & 1 & 1 & 0 & 0\\
1 & 4 & 1 & 1 & 1 & 1 & 1 & 0\\
1 & 5 & 1 & 1 & 1 & 1 & 1 & 1\\
2 & 4 & 2 & 1 & 1 & 1 & 1 & 0\\
2 & 5 & 2 & 2 & 1 & 1 & 1 & 0\\
3 & 3 & 1 & 1 & 1 & 1 & 1 & 0\\
3 & 5 & 2 & 2 & 1 & 1 & 1 & 1\\
3 & 6 & 2 & 2 & 2 & 2 & 1 & 0\\
3 & 6 & 3 & 2 & 1 & 1 & 1 & 1\\
3 & 7 & 3 & 3 & 1 & 1 & 1 & 1\\
5 & 7 & 2 & 2 & 2 & 2 & 2 & 2\\
6 & 9 & 3 & 3 & 3 & 3 & 2 & 1\\
7 & 11 & 4 & 4 & 3 & 3 & 3 & 1\\
11 & 16 & 5 & 5 & 5 & 5 & 5 & 2
\end{array}
  \end{align*}
    \end{minipage}
\end{theorem}
\begin{proof}
  We apply the method of Section \ref{effconemethod} with the following inputs:
  \begin{itemize}
    \item the base locus inequalities from Lemma \ref{lemma-baselocusexceptionals} and Lemma \ref{lemma-baselocus3pts} corresponding to the fixed divisors $E_i$, $H_1-E_i$, and $H_2-E_i-E_j-E_k$;
  \item the pullback via the morphism $X_{1,3,6} \arrow X_{1,3,5}$ of the inequalities from Theorem \ref{theorem-x135};
  \item the pullback via the restriction map $N^1(X_{1,3,6}) \arrow N^1(\widetilde{V_6})$ from Lemma \ref{lemma-restriction2} of the inequalities from Theorem \ref{theorem-yl6} cutting out $\Eff(Y_{L,6})$.
  \end{itemize}
  The Normaliz files \texttt{X136-ineqs} compute the restricted cone defined by these inequalities. Adding the fixed divisors $E_i$, $H_1-E_i$, and $H_2-E_i-E_j-E_k$ to the restricted cone and computing the extremal rays of the resulting cone, the files \texttt{X136-gens} then compute the list of extremal rays and inequalities shown in the tables above.

  The generators in Rows 1, 2, 3 and 6 are pulled back from effective classes on $X_{1,3,5}$, hence are effective.

  The generator in Row 4 is represented by the proper transform on $X_{1,3,6}$ of a hypersurface of the form $\pi_3^{-1}(Q)$ where $Q$ is a singular quadric in $\PP^3$ with vertex at $\pi_3(p_1)$ and passing through the other points $\pi_3(p_i)$. There is a unique such quadric in $\PP^3$, so this generator is represented by a fixed effective divisor.

  Finally, the generator in Row 5 is effective by Lemma \ref{vdim lower bound}.
  \end{proof}

\section{Weak Fano and log Fano varieties}\label{weak-log}
For terminology used in this section, we refer to \cite[Notation 0.4, Section 2.3]{KM98}.

A $\QQ$-factorial projective variety with finitely generated Picard group is a {\it Mori dream space} if its Cox ring
\begin{align*}
  \operatorname{Cox}(X) = \bigoplus_{L \in \Pic(X)} H^0(X,L)
\end{align*}
is a finitely generated $\CC$-algebra. Blowups of products of copies of $\PP^n$ at points that are Mori dream spaces were classified in \cite{Mukai05,CT06}. A similar question can be asked for blowups of mixed products such as the varieties $X_{m,n,s}$ and the answer is unknown in general.

Birkar--Cascini--Hacon--McKernan \cite[Corollary 1.3.2]{BCHM} showed that $X$ is a Mori dream space if it is log Fano. In particular, if $X$ is weak Fano then it is log Fano and therefore a Mori dream space. It is therefore natural to ask which of the varieties $X_{m,n,s}$ are weak Fano or log Fano. 

In this section we discuss our progress in this direction for varieties $X_{1,n,s}$ with $n=2, \, 3$.

\begin{definition}\label{definition-kltpair}
  Let $X$ be a $\QQ$-factorial variety and $\Delta$ a $\QQ$-divisor on $X$. The pair $(X,\Delta)$ is \emph{klt} if the coefficients of $\Delta$ are in the set $[0,1)$ and for any log resolution $f \colon Y \arrow X$ we have
    \begin{align*}
      K_Y+\Delta_Y &= f^*(K_X + \Delta) + \sum_i a_i E_i,
    \end{align*}
    where $\Delta_Y$ is the proper transform of $\Delta$ on $Y$, the $E_i$ are prime exceptional divisors, and $a_i >-1$ for all $i$. 
\end{definition}
\begin{definition}
  A $\QQ$-factorial projective variety $X$ is
  \begin{itemize}
  \item  \emph{log Fano} if there is a $\QQ$-divisor $\Delta$ on $X$ such that the pair $(X,\Delta)$ is klt and $-K_X-\Delta$ is ample;
    \item \emph{weak Fano} if $-K_X$ is nef and big.
  \end{itemize}
 \end{definition}
Every weak Fano variety is log Fano; a reference is \cite[Lemma 2.5]{AM16}.

Now we move on to our results. First we consider the case of threefolds. For ease of notation we will denote the anticanonical divisor simply by $-K_X$ where the variety $X_{m,n,s}$ in question is understood. We will use the fact \cite[Proposition 2.61]{KM98} that a nef divisor $D$ is big if and only if its top self-intersection satisfies $D^{\operatorname{dim} X}>0$.  

\begin{theorem}\label{theorem-weakfanox126}
The variety $X_{1,2,s}$ is weak Fano if and only if $s\le6$.
\end{theorem}

\begin{proof} 
We will show that for $s \leq 6$ the anticanonical divisor $-K_X=2H_1+3H_2-2\sum_{j=1}^sE_j$ is nef and has positive top self-intersection, while for $s \geq 7$ the top self-intersection is negative. 

Using Corollary \ref{corollary-topselfint}, we compute the top self-intersection number
\begin{align*}
  (-K_X)^3&=\left( 2H_1+3H_2 - 2 \sum_{i=1}^s E_i \right)^3\\ 
&= 2^1 \cdot 3^2 \cdot {3 \choose 1} - \sum_{i=1}^s 2^3\\
  &= 54-8s,
\end{align*}
which is positive if and only if $s\le6$. 

In order to show that $-K_X$ is nef for $s \leq 6$, we find a set which is an upper bound for its base locus, then show that it has positive degree on all curves in that set. We consider the following  unions of effective divisors, all of which belong to the anticanonical linear system:
\[
\left(H_1+H_2-\sum_{j=1}^{s-1}E_{i_j}\right)+\left(H_1+H_2-\sum_{j=1}^{s-1}E_{k_j}\right)+\left(H_2-E_a-E_b\right),
\]
where $a,b,i_j,k_j\in\{1,\dots,s\}$, the $i_j$ are distinct, as are the $k_j$, $a\ne b$, and each index appears precisely twice. That the linear systems corresponding to the three summands are nonempty follows from Lemma \ref{vdim lower bound}, indeed $\vdim|H_1+H_2-\sum_{j=1}^{s-1}E_{i_j}|=6-s$ and $\vdim|H_2-E_a-E_b|=0$. 
Therefore, the base locus of $-K_X$ is contained in the intersection of these unions of divisors, which is a subset of $\bigcup_{j=1}^sE_j$.  

Now, assume that there is an irreducible curve $C$ such that $-K_X\cdot C<0$. Then this curve must be contained in the base locus $\Bs(-K_X)\subseteq\bigcup_{j=1}^sE_j$. Now, since the $E_j$'s are disjoint, the curve $C$ must be contained in one of the exceptional divisors, say $E_j$. But for any curve $C \subset E_j$ we have $E_j \cdot C<0$, while $E_i \cdot C = 0$ for $i \neq j$ and $H_1 \cdot C = H_2 \cdot C =0$. This gives $-K_X \cdot C >0$, contrary to our assumption.  
\end{proof}
From \cite[Corollary 1.3.2]{BCHM} it follows that, for $s\le 6$, the variety $X_{1,2,s}$ is a Mori dream space. This gives a conceptual explanation for the finitely generated cones of effective divisors that we computed in Section \ref{section-dim3}.

We now turn to the case of fourfolds $X_{1,3,s}$. Here our varieties are never weak Fano, since $-K_X=2H_1+4H_2-3\sum_{i=1}^s E_i$ has negative degree on curves of the form $l_1-e_i$. We will show, however, that for up to 6 points, our varieties are still log Fano, and therefore Mori dream spaces.

Fujino--Gongyo \cite[Theorem 5.1]{FG12} showed that if $f \colon X \arrow Y$ is a surjective morphism of projective varieties and $X$ is log Fano, then $Y$ is also log Fano. Therefore it would suffice to prove our result in the case $s=6$. However, this approach would not give the log Fano structure explicitly in the case $s<6$. Instead, we give explicit log Fano structures on both $X_{1,3,5}$ and $X_{1,3,6}$. 

Since the log Fano condition requires an ample class, we start with a lemma verifying that a particular class is ample.
\begin{lemma}\label{lemma-ample}
For $s \leq 6$ the divisor class $D=2H_1+2H_2-\sum_{i=1}^s E_i$ is ample on $X_{1,3,s}$.
\end{lemma}
\begin{proof} It is sufficient to prove the case $s=6$.
 
First we claim that for any 3 distinct indices $i, \, j, \, k$ from $\{1,\ldots,6\}$, the class $H_1+H_2-E_i-E_j-E_k$ is basepoint free, hence nef. To see this, we can write the class as
  \begin{align*}
    H_1+H_2-E_i-E_j-E_k &= \left(H_1-E_i\right)+\left(H_2-E_j-E_k\right)
  \end{align*}
and by permuting indices we get 2 more such decompositions. The classes $H_1-E_i$ are disjoint for distinct $i$, while the class $H_2-E_j-E_k$ has base locus the preimage of a line in $\PP^3$ through 2 of the 3 points. Permuting indices we get 3 such lines, which are disjoint by generality. So the base locus of $H_1+H_2-E_i-E_j-E_k$ is empty as claimed.

  In particular this implies that $D$ is nef, since it can be written for example as
  \begin{align*}
  D &= \left(H_1+H_2-E_1-E_2-E_3 \right)+ \left(H_1+H_2-E_4-E_5-E_6 \right).
  \end{align*}
  By Corollary \ref{corollary-topselfint} we compute $D^4 = 4 \cdot 4 \cdot 6 - 6 >0$, so $D$ is big. Therefore it lies in the interior of the effective cone $\Eff(X_{1,3,6})$. 

 If $D$ is nef but not ample, it lies on a codimension-1 face $F$ of the nef cone $\Nef(X_{1,3,6})$. By the previous paragraph, the face $F$ intersects the interior of $\Eff(X_{1,3,6})$. So there are generators of $\Eff(X_{1,3,6})$ on both sides of $F$, and in particular we can choose a generator $G$ such that for every natural number $k$ the class $kD-G$ is not nef. We will show that for every generator $G$ of $\Eff(X_{1,3,6})$ the class $kD-G$ is nef for some $k$. By the argument above, this implies that $D$ is ample, as required.
By symmetry, it is enough to find a $k_i$ such that $k_iD-G_i$ is nef for the 6 generators $G_1,\ldots,G_6$  listed in the left-hand table of Theorem \ref{theorem-x136}. We compute:
  \begin{itemize}
  \item $D-G_1 = 2H_1+2H_2-2E_1-E_2-\cdots-E_6$: this class can be decomposed into effective classes as
    \begin{align*}
      D-G_1 &= (H_2-E_1-E_2-E_3)+(H_2-E_1-E_4-E_5)+(H_1-E_6)+H_1. 
    \end{align*}
    By permuting indices in this decomposition one can see that the base locus of $D-G_1$ consists of at most the two curves of class $l_1-e_1$ and $l_2-e_1$. But $(D-G_1) \cdot (l_1-e_1) = 0$ and $(D-G_1) \cdot (l_2-e_1)=1$, so $D-G_1$ is nef.
  \item $D-G_2 = H_1+2H_2-E_2-\cdots-E_6$: this can be decomposed into effective classes as
    \begin{align*}
      D-G_2 &= (H_1-E_2)+(H_2-E_3-E_4)+(H_2-E_5-E_6).
    \end{align*}
    By permuting indices we see that $D-G_2$ is basepoint-free, hence nef.
  \item $D-G_3 = 2H_1+H_2-E_4-E_5-E_6$: this can be decomposed into effective classes as
    \begin{align*}
      D-G_3 &= (H_1-E_4)+(H_1-E_5)+(H_2-E_6).
    \end{align*}
    By permuting indices, we see that $D-G_3$ is basepoint-free, hence nef. 
  \item  $2D-G_4 = 4H_1+2H_2-E_2-\cdots-E_6$: this class can be decomposed as $2D-G_4= 2H_1+(D+E_1)$. The proof that $D$ is nef can easily be modified to show that $D+E_1$ is nef, so $2D-G_4$ is a sum of nef classes, hence nef.
 \item $D-G_5 = H_1+H_2$ is nef.
 \item $4D-G_6= 7H_1+4H_2-E_1-\cdots-E_5-4E_6$: this class can be written as
  \begin{align*}
    4D-G_6 &= 2H_1+(H_1-E_1)+\cdots+(H_1-E_5)+4(H_2-E_6).
  \end{align*}
  The base locus of $4D-G_6$ is therefore contained in the union of the divisors $H_1-E_i$ for $i=1,\ldots,5$, together with the unique curve of class $l_1-e_6$. Each divisor $H_1-E_i$ is isomorphic to $\PP^3$ blown up in 1 point, with cone of curves spanned by $e_i$ and $l_2-e_i$. For $i=1,\ldots5$ we have $(4D-G_6) \cdot e_i =1$ and $(4D-G_6) \cdot (l_2-e_i)=3$; we also compute $(4D-G_6) \cdot (l_1-e_6) = 3$. So $4D-G_6$ is nef as required.
  \end{itemize}
\end{proof}
Now we are ready to describe the log Fano structures on our examples.
\begin{theorem}\label{theorem-logfanox135}
$X_{1,3,5}$ is log Fano.  
\end{theorem}
\begin{proof}
  There are exactly 10 planes in $\PP^3$ containing the projection images of 3 of the 5 points. Let $P_1, \ldots, P_{10}$ be the proper transforms on $X_{1,3,5}$ of the preimages of these planes. For each $P_m$, its divisor class on $X_{1,3,5}$ is of the form $H_2-E_i-E_j-E_k$. 
  Let $D_1$ and $D_2$ be the proper transforms on $X_{1,3,5}$ of two general hypersurfaces of bidegree $(1,1)$ in $\PP^1 \times \PP^3$ containing all 5 points. Each of the $D_i$ has divisor class $H_1+H_2-\sum_i E_i$.

  Now for a rational number $\epsilon$, consider the $\QQ$-divisor
  \begin{align*}
    \Delta &= \frac15 \left( \sum_{m=1}^{10} P_m \right) + (1-\epsilon)\left( D_1 + D_2 \right) + \left( \frac15 - \epsilon \right) \sum_{i=1}^5 E_i.
  \end{align*}
  For $0 < \epsilon < \frac15$ this is an effective divisor. We compute that the class of $\Delta$ in $N^1(X)$ equals
  \begin{align*}
  &  \frac15 \left( 10H_2- 6 \sum_{i=1}^5 E_i \right) + 2(1-\epsilon) \left( H_1+H_2-\sum_{i=1}^5 E_i \right) + \left( \frac15 - \epsilon \right) \sum_{i=1}^5 E_i\\
    &= (2-2 \epsilon) H_1+ (4-2\epsilon) H_2  + (-3+\epsilon) \left( \sum_{i=1}^5 E_i \right),
  \end{align*}
and so
\begin{align*}
  -K_X-\Delta &= 2\epsilon H_1 + 2 \epsilon H_2 - \epsilon \sum_{i=1}^5 E_i \\
  &= \epsilon \left( 2H_1+2H_2 - \sum_{i=1}^5 E_i \right).
\end{align*}
The class $2H_1+2H_2-\sum_i E_i$ is ample by Lemma \ref{lemma-ample}, so it remains ample when multiplied by $\epsilon >0$.

It remains to prove that $(X,\Delta)$ is a klt pair. Since we have chosen $0 < \epsilon < \frac15$, the divisor $\Delta$ is a $\QQ$-linear combination of prime divisors with all coefficients in the set $[0,1)$.
  
To compute discrepancies of the pair $(X,\Delta)$, first let us analyse the intersections among components of $\Delta$. For each point $\pi_3(p_i)$ there are 6 planes containing this point and 2 more of the points $\pi_3(p_j)$, so on $X$ there are 6 divisors $P_m$ meeting pairwise transversely along the curve $C_i = \pi_3^{-1} \pi_3 (p_i)$. Next, for two points $\pi_3(p_i)$ and $\pi_3(p_j)$, let $L_{ij} \subset \PP^3$ be the line joining them. Then there are 3 planes in $\PP^3$ containing the line $L_{ij}$ and one other point, and therefore 3 divisors $P_m$ meeting pairwise transversely along the surface $S_{ij} = \pi_3^{-1}(L_{ij})$.

Now consider the morphism
\begin{center}
  \begin{tikzcd}
    Z \arrow{r}{g} \arrow[bend right=30,swap]{rr}{h}  &Y \arrow{r}{f} &X
  \end{tikzcd}
\end{center}
where $f$ is the blowup of the curves $C_i$ and $g$ is the blowup of the proper transforms of the surfaces $S_{ij}$. One can check that $H$ is indeed a log resolution of the pair $(X,\Delta)$, so we can compute discrepancies of the pair by comparing $K_X+\Delta$ to $K_Z+\Delta_Z$. (Here $\Delta_Z$ denotes the proper transform of $\Delta$ on $Z$.)

Let $F_i$ denote the exceptional divisors of $f$ and $G_{ij}$ the exceptional divisors of $g$. Then
\begin{align*}
  K_Z = h^*K_X + 2 \sum_i F_i + \sum_{i,\, j} G_{ij}.
\end{align*}
Moreover since there are 6 planes $P_m$ through the each of the points $\pi_3(p_i)$ and 3 planes containing each of the lines $L_{ij}$, we get
\begin{align*}
   h^* \left( \sum_{m=1}^{10} P_m \right) =  \sum_{m=1}^{10} \left(P_m \right)_Z + 6 \sum_i F_i + 3 \sum_{i,j} G_{ij}.
\end{align*}
On the other hand, the divisors $D_i$ and $E_i$ do not contain the centres of the blowups in $f$ and $g$, so we have
\begin{align*}
  h^* D_i &= (D_i)_Z \quad (i=1,\, 2),\\
  h^* E_j &= (E_j)_Z \quad (j=1,\ldots,5).
\end{align*}
Putting everything together we get
\begin{align*}
  K_Z+\Delta_z = h^*(K_X+\Delta) + \frac45 \sum_i F_i + \frac25 \sum_{i,j} G_{i,j},
\end{align*}
so by Definition \ref{definition-kltpair} the pair $(X,\Delta)$ is klt.
\end{proof}
The proof in the case $s=6$ is similar. To find the required boundary divisor this time, instead of planes through 3 points we use quadric cones through all the points and with vertex at one point.
\begin{theorem}\label{theorem-logfanox136}
$X_{1,3,6}$ is log Fano.  
\end{theorem}
\begin{proof}
For $i=1,\ldots,6$, define $Q_i$ to be the proper transform of the preimage of a cone in $\PP^3$ with vertex at $\pi_3(p_i)$ and passing through the other 5 points $\pi_3(p_j)$. Now define
  \begin{align*}
    \Delta &= \frac16 \sum_i Q_i + (1-\epsilon)(D_1+D_2) +\left( \frac16 - \epsilon \right) \sum_i E_i
  \end{align*}
  For $0 < \epsilon < \frac16$ this is an effective divisor. We compute that the class of $\Delta$ in $N^1(X)$ equals
  \begin{align*}
 &\frac 16 \left( 12 H_2 - 7 \sum_{i=1}^6 E_i \right) +2(1-\epsilon)\left( H_1+H_2-\sum_{i=1}^6 E_i \right) + \left( \frac 16 - \epsilon \right) \sum_{i=1}^6 E_i\\
    &= (2-2\epsilon) H_1 +(4-2\epsilon)H_2 +(-3+\epsilon) \sum_{i=1}^6 E_i,
  \end{align*}
  so as before we have
  \begin{align*}
    -K_X-\Delta &= \epsilon \left( 2H_1 + 2 H_2 - \sum_{i=1}^6 E_i \right),
  \end{align*}
  which is again an ample class.

It remains to prove that $(X,\Delta)$ is a klt pair. Again, since we have chosen $0<\epsilon<\frac16$, the coefficients of $\Delta$ are in the set $[0,1)$. As before, let $C_i=\pi_3^{-1} \pi (p_i)$, and now let $S= \pi_3^{-1}(R)$ where $R$ is the unique twisted cubic through the 6 points $\pi_3(p_i)$. Then for any two of the quadrics, say $Q_j$ and $Q_k$, their intersection is $Q_j \cap Q_k = L_{jk} \cup R$, and moreover none of the other quadrics contains $L_{jk}$. So our log resolution is given as before by a sequence of blowups
  \begin{center}
  \begin{tikzcd}
    Z \arrow{r}{g} \arrow[bend right=30,swap]{rr}{h}  &Y \arrow{r}{f} &X
  \end{tikzcd}
  \end{center}
  where as before $f$ is the blowup of the 6 curves $C_i$ but now $g$ is the blowup of a single surface, namely the proper transform of $S$ on $Y$.

  Let $F_i$ denote the exceptional divisors of $f$ and $G$ the exceptional divisor of $g$. Then one computes that for each quadric cones $Q_i$ we have
  \begin{align*}
    h^*(Q_i) &= (Q_i)_Z + \sum_{j=1}^6 F_j + F_i + G, \quad \quad \text{so} \\
    h^*\left( \sum_i Q_i \right) &= \sum_i (Q_i)_Z + 7 \sum_i F_i + 6 G .
  \end{align*}
Putting everything together we find  
\begin{align*}
K_Z + \Delta_Z = h^*(K_X+\Delta) +\frac56 \sum_i F_i.
\end{align*}
(Note that the coefficient of the exceptional divisor $G$ on the right hand side equals 0.) So again $(X,\Delta)$ is a klt pair. 
\end{proof}
Using for example the theorem of Fujino--Gongyo mentioned above we deduce the remaining cases:
\begin{corollary}\label{corollary-logfanox13s}
For $s\leq 6$ the variety $X_{1,3,s}$ is log Fano.  
\end{corollary}

It follows that, for $s\le 6$, the variety $X_{1,3,s}$ is a Mori dream space. Again, this gives a conceptual explanation for the rational polyhedral effective cones that we found in Section \ref{section-dim4}.

\subsection*{Infinite cones of divisors}
We finish with some upper bounds on the number of points that we can blowup before cones of effective divisors cease to be finitely generated. These are based on the corresponding bounds for a single projective space: we can lift divisors from a single projective space to a product to translate these bounds to the product setting.

The main input is the following theorem of Mukai and Castravet--Tevelev \cite{Mukai05,CT06}.
\begin{theorem}\label{theorem-mukai}
  For $n>1$, let $X_{0,n,s}$ be the blowup of $\PP^n$ in a set of $s$ points in very general position. Then $\Eff(X_{0,n,s})$ is finitely generated if and only if one of the following holds:
  \begin{itemize}
  \item $n=2$ and $s \leq 8$;
  \item $n=3$ and $s \leq 7$;
  \item $n=4$ and $s \leq 8$;
  \item $n \geq 5$ and $s \leq n+3$.
  \end{itemize}
\end{theorem}
By pulling back divisors from a single projective space to a product, we get the following corollary.
\begin{corollary}\label{corollary not MDS}
  Consider a product of projective spaces
  \begin{align*}
   \PP &= \PP^{n_1} \times \cdots \times \PP^{n_k}.
  \end{align*}
Let $Y$ be the blowup of $\PP$ in a  set of $s$ points in very general position. If there exists an $n_i \neq 1$ such that $n_i$ and $s$ do not satisfy one of the conditions of Theorem \ref{theorem-mukai}, then $\Eff(Y)$ is not rational polyhedral. 
\end{corollary}
\begin{proof}
Suppose that such an $n_i$ exists; for simplicity let us denote it by $n$. Choose a $\PP^n$ factor of $\PP$ and consider the projection map $\pi \colon \PP \arrow \PP^n$. Let $p_1,\ldots,p_s$ be the points in $\PP$ and let $q_1,\ldots,q_s$ be their images in $\PP^n$.

Let $\Delta$ be an effective divisor on $X_{0,n,s}$ with class $dh-\sum_i m_i e_i$. If $d>0$ and the $m_i$ are nonnegative, then $\Delta$ corresponds to a divisor $D \subset \PP^n$ with degree $d$ and multiplicity $m_i$ at the point $q_i$. For such a divisor $D$, its preimage $\pi^{-1}(D)$ is the product of $D$ with a product of projective spaces, and therefore it has multiplicity $m_i$ at the point $p_i$. So the proper transform of $\pi^{-1}(D)$ on $Y$ has class $dH_n-\sum_i m_i E_i$ for some nonnegative integers $d$ and $m_i$. Conversely, any irreducible effective divisor on $Y$ whose class is of this form must either be the exceptional divisor over one of the points, or else pulled back from $\PP^n$ in this way. The classes of these divisors span a subcone $E_Y$ of $\Eff(Y)$ which is isomorphic to $\Eff(X_{0,n,s})$.

The subcone $E_Y$ lies in a proper face $F$ of $\Eff(Y)$, namely the face orthogonal to all curve classes in factors of $\PP$ other than $\PP^n$. No other effective divisor on $Y$ has class in the face $F$. Therefore if $\Eff(Y)$ is rational polyhedral, hence spanned by effective divisors, the subcone $E_Y$ must be a face of $\Eff(Y)$. Every face of a rational polyhedral cone is rational polyhedral, so this implies $E_Y$, equivalently $\Eff(X_{0,n,s})$, is rational polyhedral. Therefore $n$ and $s$ must satisfy one of the conditions of Theorem \ref{theorem-mukai}.
\end{proof}
The following statement summarises our knowledge of which varieties $X_{1,n,s}$ are Mori dream spaces. 

\begin{theorem} \label{theorem-summary}
  Let $X_{1,n,s}$ be the blowup of $\PP^1 \times \PP^n$ in $s$ points in very general postions. Then $X_{1,n,s}$ is a Mori dream space in the following cases:
  \begin{enumerate}
  \item[(a)] $n=2$ or $n=3$ and $s \leq 6$;
    \item[(b)] $n$ arbitrary and $s \leq n+1$.
  \end{enumerate}
  On the other hand, $X_{1,n,s}$ is not a Mori dream space in the following cases:
  \begin{enumerate}
  \item[(c)] $n=2$ or $n=4$ and $s \geq 9$;
  \item[(d)] $n=3$ and $s \geq 8$;
    \item[(e)] $n \geq 5$ and $s \geq n+4$.
  \end{enumerate}
\end{theorem}
\begin{proof}
  Statement $(a)$ was proved in Theorem \ref{theorem-weakfanox126} and Corollary \ref{corollary-logfanox13s}. Statement $(b)$ follows from the theorem of Hausen--S{\"u}{\ss} \cite[Theorem 1.3]{HS10}.

  Statements $(c)$, $(d)$, $(e)$ follow from Theorem \ref{theorem-mukai} and Corollary \ref{corollary not MDS} by taking $k=2$ and $n_1=1$.
\end{proof}

%% {\color{blue}
%% \begin{proof}[Proof of Theorem \ref{thmMDS}]
%% It follows from Corollary \ref{corollary not MDS} that
%%  $X_{1,2,s}$ is not a Mori dream space for $s\ge 9$, and 
%% $X_{1,3,s}$ is not a Mori dream space for $s\ge 8$. This, together with Theorem \ref{theorem-weakfanox126} and Corollary \ref{corollary-logfanox13s} conclude the proof.
%% \end{proof}}

%% For blowups of $\PP^1 \times \PP^2$ or $\PP^1 \times \PP^3$, the only open cases are the blowup of 7 or 8 points in $\PP^1 \times \PP^2$ and the blowup of 7 points in $\PP^1 \times \PP^3$. We pose the following question.

These results leave only a small number of open cases in each dimension, which we address in the following questions.

\begin{question}\label{question1}
Are the varieties $X_{1,2,7}$, $X_{1,2,8}$, $X_{1,3,7}$ log Fano or Mori dream spaces?
\end{question}

%% In higher dimensions, the theorem of Hausen--S\"u{\ss} implies that $X_{1,n,n+1}$ is always a Mori dream space, while Corollary \ref{corollary not MDS} implies that, for $n=4$ and $s \geq 9$ or $n \geq 5$ and $s \geq n+4$, the varieties $X_{1,n,s}$ are not Mori dream spaces. This suggests the following questions:  
\begin{question}\label{question2} For $s=6,\, 7,\, 8$, are the varieties $X_{1,4,s}$ log Fano or Mori dream spaces?
\end{question}
The methods of this paper could in principle be applied to study these two questions. It would be interesting to know if our methods can be applied successfully in these cases.

Finally, in higher dimensions, the remaining open cases are the following:
\begin{question}\label{question3}
For $n \geq 5$, are $X_{1,n,n+2}$ and $X_{1,n,n+3}$ log Fano or Mori dream spaces?
\end{question}

\bibliographystyle{alpha}
\bibliography{biblio}
%% {\small
%% {\sc Department of Mathematical Sciences, Loughborough University, Epinal Way, Loughborough LE11 3TU, United Kingdom}

%% \noindent {\tt A.Prendergast-Smith@lboro.ac.uk}

%% \noindent {\tt T.Grange@lboro.ac.uk}
%% \newline

%% \noindent {\sc Dipartimento di Matematica, Universit\`a degli Studi di Trento, via Sommarive 14, I-38123 Povo di Trento (TN), Italy}

%% \noindent{\tt elisa.postinghel@unitn.it}
%% }
\end{document}